\documentclass[3p]{elsarticle}

\usepackage{amsmath}
\usepackage{amsthm}
\usepackage{amssymb}
\usepackage{bm}
\usepackage{accents}
\usepackage{color}
\usepackage{nicematrix}
\usepackage{booktabs}  
\usepackage[pdftex,plainpages=false,colorlinks=true,citecolor=RoyalBlue,linkcolor=Red,urlcolor=black,filecolor=black,bookmarksopen=true]{hyperref}

\newtheorem{thm}{Theorem}

\newdefinition{rmk}{Remark}
\newproof{pf}{Proof}

\numberwithin{equation}{section}

\graphicspath{{figs/}}

\begin{document}

\begin{frontmatter}
	
\title{Computational homogenization of unsteady flows in a periodic porous medium\tnoteref{label1}}
\tnotetext[label1]{The work was supported by the Russian Science Foundation (grant No. 24-11-00058).}

\author{P.N. Vabishchevich\fnref{lab1,lab2}}
\ead{vab@cs.msu.ru}

\address[lab1]{Lomonosov Moscow State University, 1, building 52, Leninskie Gory, 119991 Moscow, Russia}

\address[lab2]{North Caucasus Federal University, 1, Pushkin str., 355017 Stavropol, Russia}

\begin{abstract}

The work is devoted to the development and computational implementation of the homogenization method for modeling unsteady flows of a viscous incompressible fluid in periodic porous media taking into account memory effects. At the macrolevel, the flow is described by an integro-differential Darcy law with a tensor memory kernel determined by solving unsteady problems on the periodicity cell.
The developed approach to computational homogenization is based on finding the steady-state and unsteady components of the conductivity tensor from solving auxiliary boundary value and spectral problems on the periodicity cell. The nonlocal macroscopic problem is transformed into a local system of differential equations by approximating the memory kernel as a sum of exponentials.
Issues of spatial finite element approximation are discussed, and stable two-level schemes in time are constructed. The results of applying the developed computational homogenization technology for unsteady filtration problems in porous media to a two-dimensional test problem are presented.

\end{abstract}

	
\end{frontmatter}

\section{Introduction}

The prediction and monitoring of groundwater, the development of oil fields, and their optimization are impossible without a quantitative description of filtration processes \cite{bear2013dynamics, anderson2015applied}.  
Geological media are heterogeneous across many scales.  
Consequently, direct numerical simulation of individual pores at the spatial scales of interest remains unattainable even with modern high-performance computing systems. It is necessary to have physically sound and computationally efficient macroscopic models that retain key features of the flows at the microscale, such as retardation and dispersion.

At the pore scale, the fluid is described by the classical Navier–Stokes equations.  
For slow flows, inertial terms are small, and the system simplifies to the linear Stokes equations for incompressible flows.  
When averaging over a volume containing many pores, the microscopic velocity field is regarded as the filtration velocity, and the pressure gradient as the hydraulic gradient.  
Such a mathematical model of flows at the macrolevel is known as Darcy's law \cite{bear2010modeling, allen2021mathematics}.

Rigorous theoretical models for constructing macroscale models are based on homogenization theory \cite{sanchez1980non, auriault2010homogenization}.  
Under the assumption of a periodic microstructure, the method of two-scale asymptotic expansions is applied.  
A small dimensionless parameter $\varepsilon$ is introduced, defined as the ratio of the characteristic pore size to the size of the entire domain, assuming periodicity of the pore space geometry.  
The asymptotic expansion of the velocity and pressure fields in powers of $\varepsilon$ leads to boundary value problems solved on a periodicity cell.  
Averaging over the periodicity cell yields a symmetric positive definite absolute permeability tensor relating the filtration velocity and the pressure gradient via Darcy's law \cite{hornung1997homogenization, mikelic2007homogenization}.  
This approach ensures rigorous mathematical correctness and allows one to systematically account for the influence of microstructure on macroscopic properties without empirical corrections.

When the characteristic time of viscous diffusion in the pores is comparable to the observation time, the macroscopic response becomes inertial and depends on the history of the process.  
Mathematical studies of homogenization of unsteady flows in porous media began with the work \cite{lions1981some}.  
Systematic two-scale analysis leads to an integro-differential Darcy law, in which the filtration velocity at the current time is determined by the convolution of the pressure gradient with a tensor memory kernel \cite{allaire1992homogenization, mikelic1994mathematical, sandrakov1997homogenization}.  
The components of the tensor memory kernel are determined by averaging over the pore volume of the solution of the corresponding unsteady Stokes boundary value problem on the cell.

In the computational homogenization of unsteady filtration problems, we must address two main issues.  
The first concerns the calculation of the memory kernel and poses no major computational difficulties: it is necessary to solve unsteady problems on the periodicity cell and average the solution. For this, we can use, for example, standard finite element or finite volume computational technologies \cite{KnabnerAngermann2003, QuarteroniValli1994}.

The second issue, which deserves more attention, is related to the approximate solution of boundary value problems for integro-differential equations.  
Computational algorithms must take into account the most important features of the problems under consideration, which arise from the fact that the kernels are difference kernels and decay exponentially in time. We can rely on the use of the finite element method for spatial approximation and implicit time approximations using appropriate quadrature formulas for the integral terms \cite{ChenBook1998}.

The computational complexity of solving problems with memory stems from the need to work with the approximate solution over many time steps.  
Accounting for the specifics of integro-differential equations with difference kernels is ensured, in particular, by using convolution quadratures \cite{lubich1988convolution, jin2023numerical}. This allows one to update the solution recursively with linear complexity in the number of steps, avoiding the need to store the solution at all previous time instants.

A more promising approach for computational practice is related to using special approximations of the memory kernel \cite{linz1985analytical}.  
In this regard, we note the possibility of approximately solving problems for integro-differential equations by approximating the memory kernel as a finite sum of exponentials.  
In this case, a transformation from a nonlocal problem to a local one is achieved, where memory effects are accounted for by an additional system of ordinary differential equations.  
Such a computational technology in various forms is explored, in particular, in works \cite{vabishchevich2022numerical, vabishchevich2023approximate}.

In light of the above, the computational implementation of homogenization for unsteady flows in porous media is carried out in three stages: (i) solving a set of problems for unsteady equations in the periodicity cell to determine the components of the dynamic permeability tensor; (ii) approximating the components of the tensor memory kernel as a sum of exponentials; (iii) solving the local problem at the macrolevel for fluid filtration taking into account memory effects.  
The original equation is augmented by a system of linear ordinary differential equations, which does not introduce any fundamental complications. Consequently, we can construct computational algorithms for solving filtration problems with memory by modifying the methods used for standard filtration problems.

The approximate solution of the problems at stage (i) can be carried out directly based on the unsteady Stokes equations \cite{roychowdhury, glowinski2022numerical}. Also worth noting is a second option that can be used for the approximate solution of the integro-differential equation. In this case, stationary Stokes problems are solved for the Laplace transform with a complex parameter, after which the inverse Laplace transform in time is applied to the tensor memory kernel.

In implementing stage (ii), various computational algorithms for function approximation can be used \cite{meinardus2012approximation, trefethen2019approximation}.  
When approximating by a sum of exponentials, we note two possibilities. The first is related to approximating the kernel itself.  
In this case, we can rely on the Prony algorithm and its generalizations \cite{pereyra2010exponential}.  
We also highlight algorithms for nonlinear approximations taking into account the sign-definiteness of the coefficients of exponential approximations based on the non-negative least squares (NNLS) algorithm \cite{holmstrom2002review,vabishchevich2026computational}.  
The second possibility is related to rational approximations of the image of the function when applying the Laplace transform.  
The problem of nonlinear approximations by rational functions is better developed \cite{braess1986nonlinear}. Currently, simple and stable Adaptive Antoulas–Anderson algorithms are widely used \cite{nakatsukasa2025applications}.

In computational homogenization of unsteady problems, stages (i) and (ii) can be replaced, as was done in \cite{vab2025hom} for diffusion problems.  
Instead of solving unsteady problems on the periodicity cell, averaging the solution over the cell, and approximating the memory kernels by a sum of exponentials, a partial eigenvalue and eigenfunction problem is solved, extracting the smallest positive eigenvalues. This approach seems promising for modern computational practice in homogenization of unsteady processes. The aim of this study is to apply such a developed computational homogenization technology to problems of unsteady filtration in porous media.

The organization of our work is as follows.  
Section 2 formulates the homogenization problem for describing unsteady flows of a viscous incompressible fluid in porous media with a periodic structure. Memory effects are described by a nonlocal Darcy law with dynamic permeability.  
In Section 3, the nonlocal problem is transformed into a local one based on approximating the components of the tensor memory kernel as a sum of exponentials by solving spectral problems on periodicity cells.  
Section 4 discusses issues of the computational implementation of the proposed approach based on finite element spatial approximations. Two-level time approximations are constructed and analyzed for stability for the approximate solution of the macroscale problem.  
Section 5 presents the key elements of the proposed algorithm for averaging unsteady filtration processes using a two-dimensional test problem as an example.  
Section 6 summarizes the findings of our study.

\section{Problem statement}\label{sec:2}

\begin{figure}[ht]
	\center
	\includegraphics[width=1\linewidth]{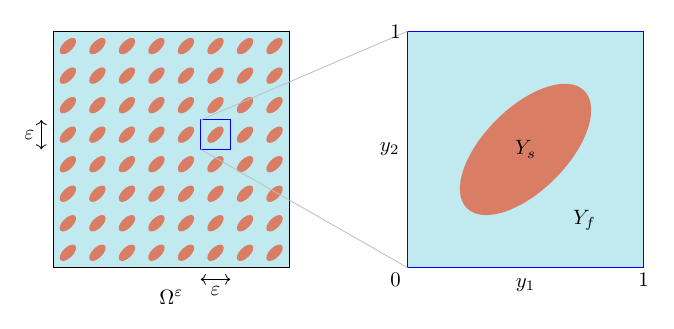}
	\caption{Periodic porous medium and periodicity cell}
	\label{f-1}
\end{figure}

We consider a periodic porous medium with a characteristic heterogeneity scale $\varepsilon \ll 1$ in the domain $\Omega \subset \mathbb{R}^d$ ($d=2,3$). The pore space itself is assumed to be connected. The flow is modeled in the part $\Omega^\varepsilon$ of the domain occupied by pores.
In homogenization theory, a fast variable $\bm{y} = \bm{x}/\varepsilon \in Y = (0,1)^d$ is introduced, where $Y$ is the elementary cell of the periodic structure. The cell consists of two subdomains: $Y = Y_f \cup Y_s \cup \Gamma$, where $Y_f$ is the fluid part, $Y_s$ is the solid part, and $\Gamma = \partial Y_f \cap \partial Y_s$ (Fig.~\ref{f-1}).

Slow unsteady flows of an incompressible fluid in the pore space are described by the unsteady Stokes system for the velocity $\bm u^\varepsilon(\bm x,t)$ and pressure $p^\varepsilon(\bm x,t)$, which after appropriate nondimensionalization takes the form:
\begin{equation}\label{2.1}
	\frac{\partial \bm u^\varepsilon}{\partial t} 
    - \varepsilon^2 \Delta\bm u^\varepsilon +\nabla p^\varepsilon =\bm  f(\bm  x, t), 
\end{equation}
\begin{equation}\label{2.2}
	\nabla\cdot \bm u^\varepsilon = 0, \quad \bm x \in \Omega^\varepsilon, \ t > 0,
\end{equation}
where $\bm f(\bm x, t)$ is a given body force.
We assume that the boundary and initial conditions are given by
\begin{equation}\label{2.3}
	\bm u^\varepsilon(\bm x, t) = 0, \quad \bm x \in \partial \Omega^\varepsilon, \ t > 0.
\end{equation}
\begin{equation}\label{2.4}
	\bm u^\varepsilon(\bm x, 0) = 0, \quad \bm x \in \Omega^\varepsilon.
\end{equation}

At the macroscale, homogenization theory \cite{allaire1992homogenization, mikelic1994mathematical, sandrakov1997homogenization} yields Darcy's law with memory:
\begin{equation}\label{2.5}
	\bm u(\bm x, t) = -\int_0^t K(t-s)\bigl(\nabla p(\bm x,s) - \bm f(\bm x, s)\bigr) \,ds .
\end{equation}
For the incompressible media under consideration, the continuity equation is employed:
\begin{equation}\label{2.6}
	\nabla\cdot \bm u(\bm x, t) = 0, \quad \bm x \in \Omega, \ t > 0.
\end{equation}
The components of the dynamic permeability tensor $K\colon\mathbb R^{d\times d}\to\mathbb R^{d\times d}$ are computed by solving unsteady problems on the periodicity cell.

When homogenizing the unsteady problems \eqref{2.1}--\eqref{2.4}, we have two similar possibilities \cite{wang2025corrector} for the cell problems.
In the first case \cite{allaire1992homogenization}, we consider problems with a constant right-hand side.
The second option \cite{mikelic1994mathematical}, which is used in our work, is associated with problems with a constant initial condition.

The dynamic permeability tensor $K = (K_{ij})$ is computed by solving the auxiliary problems $\bm w_j(\bm y,t),$ $\pi_j(\bm y,t), \ j =1,2, \ldots, d$ on the cell:
\begin{equation}\label{2.7}
  \frac{\partial \bm w_j}{\partial t} 
  - \Delta_{\bm y}\bm w_j+\nabla_{\bm y}\pi_j = 0,
  \quad \bm y \in Y_f, \ t > 0,
\end{equation}
\begin{equation}\label{2.8}
  \nabla_{\bm y}\cdot \bm w_j = 0, \quad \bm y \in Y_f, \ t > 0,
\end{equation}
\begin{equation}\label{2.9}
  \bm w_j(\bm y, t) = 0, \quad \bm y \in \Gamma, \ t > 0,
\end{equation}
\begin{equation}\label{2.10}
  \bm w_j(\bm y, 0) = \bm e_j, \quad \bm y \in Y_f, \ t > 0,
\end{equation}
where $\bm{e}_j$ are the unit vectors of the Cartesian coordinate system. The components of the permeability tensor are determined from the solutions of problems \eqref{2.7}--\eqref{2.10} by the formulas:
\begin{equation}\label{2.11}
 K_{ij}(t)=\int_{Y_f}\bm w_j(\bm y, t) \cdot \bm e_i \,d\bm y ,
 \quad i, j = 1,2, \ldots, d.
\end{equation}
The fundamental properties of the permeability tensor hold: symmetry, positive definiteness, and exponential decay in time.

It is necessary to develop a computational algorithm for solving the macroscale problem \eqref{2.5}, \eqref{2.6} based on computing the dynamic permeability tensor according to \eqref{2.11} from the solutions of the periodicity cell problems \eqref{2.7}--\eqref{2.10}.

\section{Transformation of the problem} \label{sec:3}

The main difficulties in the numerical solution of the considered computational homogenization problem are related to the necessity of solving a problem nonlocal in time. The solution for the pressure, determined from the integro-differential equations \eqref{2.5} and \eqref{2.6}, depends on its entire history at each new time level.  
Consequently, the natural approach (see, e.g., \cite{ChenBook1998}), based on applying suitable quadrature formulas, leads to an unacceptable increase in computational cost both in time and memory requirements.
An additional difficulty is that the tensor kernel $K(t)$ is not given explicitly but is determined according to \eqref{2.7}--\eqref{2.11} from the solutions of auxiliary unsteady problems on the periodicity cell.

Following \cite{vabishchevich2022numerical,vabishchevich2023approximate}, we focus on computational algorithms that account for memory effects at a lower cost.  
To this end, for the approximate solution we transition from a nonlocal problem to a local one by introducing additional unknowns and extending the original system of equations.

Similarly to \cite{vab2025hom}, we associate the following auxiliary spectral problems with the periodicity cell problems \eqref{2.7}--\eqref{2.10}:
\begin{equation}\label{3.1}
  - \Delta_{\bm y}\bm \varphi+\nabla_{\bm y}\eta = \lambda \bm \varphi, \quad \bm y \in Y_f, 
\end{equation}
\begin{equation}\label{3.2}
  \nabla_{\bm y}\cdot \bm \varphi = 0, \quad \bm y \in Y_f,
\end{equation}
\begin{equation}\label{3.3}
  \bm \varphi(\bm y) = 0, \quad \bm y \in \Gamma,
\end{equation}
with additional periodicity conditions.
The eigenvalues and eigenfunctions $\lambda_k, \bm \varphi_k(y)$, $k = 1,2, \ldots$, of problems \eqref{3.1}--\eqref{3.3} are real and satisfy the inequalities \cite{galdi2011introduction}:
\begin{equation}\label{3.4}
	0 < \lambda_1 \leq \lambda_2 \leq \ldots \, .
\end{equation}
For the solution of problem \eqref{2.7}--\eqref{2.10}, we obtain
\begin{equation}\label{3.5}
	\bm w_j(y,t) = \sum_{k=1}^{\infty} (\bm e_j, \bm \varphi_k) \exp(-\lambda_k t) \bm \varphi_k(y), 
	\quad \bm y \in Y_f,\ t > 0
    \quad j = 1,2, \ldots, d.
\end{equation}
Here $(\cdot, \cdot)$ for vector functions denotes the inner product in the space $\bm L_2(Y_f)$, so that
\[
(\bm e_j, \bm \varphi_k) = \int_{Y_f} \bm e_j \cdot \bm \varphi_k(y) \, dy, 
\quad k = 1,2, \ldots,
\quad j = 1,2, \ldots, d.
\]

Taking into account \eqref{2.11} and \eqref{3.5}, for the components of the permeability tensor we have
\begin{equation}\label{3.6}
 K_{ij}(t)= \sum_{k=1}^{\infty} a_i^k a_j^k \exp(-\lambda_k t) ,
\end{equation}
where
\[
 a_i^k = (\bm e_i, \bm \varphi_k),
 \quad k = 1,2, \ldots,
 \quad i = 1,2, \ldots, d .
\]
Thus, the components of the permeability tensor are expressed as a sum of exponentials, and the tensor itself is symmetric and positive semidefinite:
\[
  K_{ij}(t) = K_{ji}(t),
  \quad i, j = 1,2, \ldots, d,
  \quad K(t) \ge 0 .
\]
Positive definiteness should be established by a more detailed study.

Without accounting for memory effects, the permeability tensor $\accentset{-}{K}_{ij}$ is determined by solving boundary value problems on the periodicity cell:
\begin{equation}\label{3.7}
  - \Delta_{\bm y}\bm w_j^0+\nabla_{\bm y}\pi_j = \bm e_j,
  \quad \bm y \in Y_f,
\end{equation}
\begin{equation}\label{3.8}
  \nabla_{\bm y}\cdot \bm w_j^0 = 0, \quad \bm y \in Y_f, 
\end{equation}
with periodicity conditions and
\begin{equation}\label{3.9}
  \bm w_j^0(\bm y) = 0, \quad \bm y \in \Gamma ,
  \quad j = 1,2, \ldots, d .
\end{equation}
The constant permeability tensor is found analogously to \eqref{2.11}:
\begin{equation}\label{3.10}
 \accentset{-}{K}_{ij} =\int_{Y_f}\bm w_j^0(\bm y) \cdot \bm e_i \,d\bm y ,
 \quad i, j = 1,2, \ldots, d.
\end{equation}
Using the solution of the spectral problem \eqref{3.1}--\eqref{3.2} and \eqref{3.6}, a connection is established between the permeability tensor $\accentset{-}{K}_{ij}$ and the tensor memory kernel $K_{ij}(t)$:
\begin{equation}\label{3.11}
 \accentset{-}{K}_{ij} = \int_0^\infty K_{ij}(t) \, d t = \sum_{k=1}^{\infty} \frac {a_i^k a_j^k} {\lambda_k} ,
\quad i, j = 1,2, \ldots, d .
\end{equation}
Taking into account the positive definiteness of $\accentset{-}{K}_{ij}$, it follows from \eqref{3.6} that the tensor memory kernel $K_{ij}(t)$ is positive definite for $t \ge 0.$

Our computational homogenization algorithm is based on representing the tensor memory kernel as a sum of exponentials \eqref{3.6}. In this case, we have the possibility of transitioning from a nonlocal problem to a local one.
A natural approach involves approximating the elements of the tensor by a finite sum, retaining only the first $m$ eigenvalues.
In representation \eqref{3.6}, we separate two terms:
\[
 K_{ij}(t)=  K_{ij}^m(t) + \delta K_{ij}^m(t),
\]
where
\[
K_{ij}^m(t)  = \sum_{k=1}^{m} a_i^k a_j^k \exp(-\lambda_k t), 
\quad \delta K_{ij}^m(t)  = \sum_{k=m+1}^{\infty} a_i^k a_j^k \exp(-\lambda_k t), \quad i, j = 1,2, \ldots, d .
\]
The influence of the terms $\delta K_{ij}^m(t), \ i, j = 1,2, \ldots, d$ is taken into account approximately. 
For sufficiently large $m$, the eigenvalues $\lambda_k$, $k = m+1, m+2, \ldots$ are large. The function
\[
s(t) = 
\begin{cases}
\lambda \exp(-\lambda t), & t \geq 0, \\
0, & t < 0
\end{cases}
\]
tends to the delta function $\delta(t)$ as $\lambda \rightarrow \infty$. Taking this into account, we obtain
\[
\delta K_{ij}^m(t) \approx \widetilde{ K}_{ij}^m \delta(t),
\quad \widetilde{ K}_{ij}^m = 
\sum_{k=m+1}^{\infty} \frac {a_i^k a_j^k} {\lambda_k} ,
\quad i, j = 1,2, \ldots, d .
\]
Considering \eqref{3.11}, we have the approximate equality
\begin{equation}\label{3.12}
 K_{ij}(t) \approx  K_{ij}^m(t) + \widetilde{ K}_{ij}^m \delta(t),
\end{equation}
in which
\begin{equation}\label{3.13}
\widetilde{ K}_{ij}^m = \accentset{-}{K}_{ij} - \sum_{k=1}^{m} \frac {a_i^k a_j^k} {\lambda_k} ,
\quad i, j = 1,2, \ldots, d .
\end{equation}
This makes it possible to calculate the components of the tensor $\widetilde{ K}_{ij}^m$ from the first $m$ eigenvalues and eigenfunctions $\lambda_k, \bm \varphi_k(y)$ when determining $\accentset{-}{K}_{ij}$ from \eqref{3.7}--\eqref{3.10}.

The features of the macroscale problems are demonstrated for a constant body force $\bm f(\bm x)$. Considering the general case $\bm f(\bm x,t)$ does not cause fundamental complications. Taking into account \eqref{2.5}, \eqref{2.6} and \eqref{3.12}, the corresponding approximate solution $\widetilde{p}(x,t)$ is determined from the equation
\begin{equation}\label{3.14}
- \nabla \cdot \left ( \widetilde{ K}^m \nabla \widetilde{p} \right )
- \nabla \cdot \int_0^t K^m (t-s) \nabla \widetilde{p} (\bm x, s) d s = \psi(\bm x,t),
\quad \bm x \in \Omega,\ t > 0 .
\end{equation}
The right-hand side, given our assumptions, has the form
\begin{equation}\label{3.15}
\psi(\bm x,t) = - \nabla \cdot \bm g,
\quad 
\bm g(\bm x,t) = \accentset{-}{K} \bm f(\bm x)
- \varPhi (t) \bm f(\bm x) ,
\end{equation}
where the unsteady part is determined by the tensor $\varPhi(t)$ with components
\[
\varPhi_{ij}^m(t)  = \sum_{k=1}^{m} \frac{a_i^k a_j^k} {\lambda_k} \exp(-\lambda_k t), 
\quad i, j = 1,2, \ldots, d .
\]
Taking into account \eqref{2.3}, equation \eqref{3.14} is supplemented with the boundary condition
\begin{equation}\label{3.16}
\left ( \widetilde{ K}^m \nabla \widetilde{p}
+ \int_0^t K^m (t-s) \nabla \widetilde{p} (\bm x, s) d s  \right ) \bm \nu =  g_\nu (\bm x, t) ,
\quad \bm x \in \partial \Omega,\ t > 0 ,
\end{equation}
in which
\[
 g_\nu (\bm x, t) = \bm g(\bm x, t) \cdot \bm \nu ,
\]
where $\bm \nu$ is the outward normal to $\Omega$. The solution is sought taking into account that the pressure is defined up to a constant.

To formulate the local problem, similarly to \cite{vabishchevich2022numerical, vabishchevich2023approximate}, we introduce $m$ auxiliary functions:
\begin{equation}\label{3.17}
c_k(x,t) = \int_0^t \exp(-\lambda_k (t-s)) \widetilde{p}(\bm x,s) ds,
\quad k=1,2,\ldots,m.
\end{equation}
Using these functions, equation \eqref{3.14} can be rewritten as
\begin{equation}\label{3.18}
- \nabla \cdot \left ( \widetilde{ K}^m \nabla \widetilde{p} \right )
- \sum_{k=1}^m \nabla \cdot \left ( D^k \nabla c_k \right ) = \psi(\bm x,t),
\quad \bm x \in \Omega,\ t > 0 ,
\end{equation}
where for the elements of the tensors $D^k, \ k=1,2,\ldots,m$ we have
\[
 D^k_{ij} = a_i^k a_j^k ,
 \quad i, j = 1,2, \ldots, d .
\]
The boundary condition \eqref{3.16} yields
\begin{equation}\label{3.19}
\left ( \widetilde{ K}^m \nabla \widetilde{p}
+ \sum_{k=1}^m D^k \nabla c_k \right ) \bm \nu =   g_\nu (\bm x, t),
\quad \bm x \in \partial \Omega,\ t > 0 .
\end{equation}
The introduced auxiliary functions $c_k(x,t)$, $k = 1,2,\ldots,m$, are determined from the system of ordinary differential equations:
\begin{equation}\label{3.20}
\frac{\partial c_k}{\partial t} + \lambda_k c_k - \widetilde{p} = 0,
\quad k = 1,2,\ldots,m.
\end{equation}
Taking into account \eqref{3.17}, the initial conditions are formulated as follows:
\begin{equation}\label{3.21}
c_k(x,0) = 0,
\quad x \in \Omega,
\quad k = 1,2,\ldots,m.
\end{equation}

The boundary value problem \eqref{3.18}--\eqref{3.21} is considered under the natural assumption of positive definiteness of the tensor $\widetilde{ K}^m$. We present an a priori estimate for the solution, which ensures the well-posedness of our problem.

\begin{thm}\label{t-1}
For the solution of problem \eqref{3.18}--\eqref{3.21} with $\widetilde{ K}^m > 0$, the following a priori estimate holds:
\begin{equation}\label{3.22}
 \int_{0}^{T} \left ( \widetilde{ K}^m \nabla \widetilde{p}, \nabla \widetilde{p} \right ) d t +
 \sum_{k=1}^m \left (D^k \nabla c_k , \nabla c_k \right ) 
 \leq \int_{0}^{T} \left ( \left (\widetilde{ K}^m \right )^{-1}   
 \bm g, \bm g \right ) d t ,
	\quad t > 0.
\end{equation}
\end{thm}

\begin{proof}
From \eqref{3.20} we have
\begin{equation}\label{3.23}
c_k = \frac{1}{\lambda_k} \left (\widetilde{p} - \frac{\partial c_k}{\partial t} \right ) ,
\quad k = 1,2,\ldots,m.
\end{equation}
Substitution into \eqref{3.18} taking into account \eqref{3.13} leads to the equation
\begin{equation}\label{3.24}
- \nabla \cdot \left ( \widetilde{ K}^m \nabla \widetilde{p} \right )
- \sum_{k=1}^m \nabla \cdot \left ( \widetilde{D}^k \nabla \left (\widetilde{p} - \frac{\partial c_k}{\partial t} \right ) \right ) = \psi(\bm x,t),
\quad \bm x \in \Omega,\ t > 0 ,
\end{equation}
where for the elements of the tensors $\widetilde{D}^k, \ k=1,2,\ldots,m$ we have
\[
 \widetilde{D}^k_{ij} = \frac{a_i^k a_j^k}{\lambda_k} ,
 \quad i, j = 1,2, \ldots, d .
\]
Multiply equation \eqref{3.24} scalarly by $\widetilde{p}.$ Taking into account the boundary condition \eqref{3.19}, we obtain
\begin{equation}\label{3.25}
\left ( \widetilde{ K}^m \nabla \widetilde{p}, \nabla \widetilde{p} \right )
+ \sum_{k=1}^m \left ( \widetilde{D}^k \nabla \left (\widetilde{p} - \frac{\partial c_k}{\partial t} \right ), \nabla \widetilde{p} \right ) = (\bm g, \nabla \widetilde{p}) .
\end{equation}
From \eqref{3.23} it follows that
\[
D^k \nabla c_k - \widetilde{D}^k \nabla \left (\widetilde{p} - \frac{\partial c_k}{\partial t} \right ) = 0,
\quad k = 1,2,\ldots,m.
\]
Multiplying each equation scalarly by $\displaystyle{\nabla \frac{\partial c_k}{\partial t} }$ and summing them yields:
\begin{equation}\label{3.26}
 \sum_{k=1}^m \left (D^k \nabla c_k , \nabla \frac{\partial c_k}{\partial t} \right ) -
\sum_{k=1}^m \left ( \widetilde{D}^k \nabla \left (\widetilde{p} - \frac{\partial c_k}{\partial t} \right ), \nabla \frac{\partial c_k}{\partial t}  \right ) = 0.
\end{equation}
From \eqref{3.25} and \eqref{3.26} the inequality follows:
\begin{equation}\label{3.27}
\left ( \widetilde{ K}^m \nabla \widetilde{p}, \nabla \widetilde{p} \right )
+ \frac 12 \frac{d} {d t} \sum_{k=1}^m \left (D^k \nabla c_k , \nabla c_k \right ) \leq (\bm g, \nabla \widetilde{p}) .
\end{equation}
Taking into account the inequality
\[
 (\bm g, \nabla \widetilde{p}) \leq \frac 12 \left ( \widetilde{ K}^m \nabla \widetilde{p}, \nabla \widetilde{p} \right )
 + \frac 12 \left ( \left (\widetilde{ K}^m \right )^{-1}  \bm g, \bm g \right ) ,
\]
from \eqref{3.27} we obtain
\[
\left ( \widetilde{ K}^m \nabla \widetilde{p}, \nabla \widetilde{p} \right )
+ \frac{d} {d t} \sum_{k=1}^m \left (D^k \nabla c_k , \nabla c_k \right ) \leq \left ( \left (\widetilde{ K}^m \right )^{-1}  \bm g, \bm g \right ) .
\]
Integrating with respect to $t$ from $0$ to $T$ and taking into account the initial conditions \eqref{3.21}, we arrive at the desired estimate \eqref{3.22}.
\end{proof}

\section{Computational algorithm} \label{sec:4}

We single out the individual problems that are solved in the considered computational homogenization. There are three such problems:

\begin{enumerate}
\item Boundary value problems \eqref{3.7}--\eqref{3.9} on the periodicity cell $Y_f$ for determining the steady-state permeability tensor according to \eqref{3.10}.
\item Solving the partial eigenvalue problem \eqref{3.1}--\eqref{3.3} $(\lambda_k, \bm \varphi_k(\bm y)$, $k = 1,2,\ldots,m)$ in the domain $Y_f$ and calculating the components of the steady-state tensor $\widetilde{ K}^m$ according to \eqref{3.13} and the unsteady tensor $K^m(t)$ in representation \eqref{3.12}.
\item Solving problem (\ref{3.18})--(\ref{3.21}) to find the homogenized solution in the computational domain $\Omega$.
\end{enumerate}

The numerical solution of the steady-state Stokes problems \eqref{3.7}--\eqref{3.9} is performed using finite element approximations on triangular (tetrahedral) meshes. In computational practice for modeling incompressible fluid flows, the Hood–Taylor finite elements are most widely used \cite{roychowdhury,glowinski2022numerical}, where the velocity is approximated using second-degree polynomials, and the pressure using first-degree polynomials.

For the approximate computation of the memory kernels (components of the conductivity tensor), we first solve the partial spectral problem \eqref{3.1}--\eqref{3.3} in the periodicity cell.
The features of the considered spectral problems \cite{Babuska1991} when using finite elements are taken into account, first of all, in the construction of the discrete problem (mixed eigenvalue problems) \cite{boffi2013mixed}. In the numerical solution of the corresponding algebraic problem (saddle point problems), well-developed computational algorithms \cite{saad2011numerical} are used to find the first minimal (in absolute value) eigenvalues and eigenfunctions of generalized eigenvalue problems.

We pay the greatest attention to the computational algorithm for the approximate solution of the macroscale problem \eqref{3.18}--\eqref{3.21}. In the computational domain $\Omega$, a triangular (tetrahedral for three-dimensional problems) mesh is introduced. When approximating with linear finite elements on this mesh, we define the space of functions $V^h \subset H^1(\Omega)$. We seek an approximate solution of problem \eqref{3.18}--\eqref{3.21} $v(\bm x,t) \in V^h, \ v_k(\bm x,t) \in V^h, \ k = 1,2,\ldots,m$, corresponding to $\widetilde{p}(x,t)$ and $c_k(\bm x,t), \ k = 1,2,\ldots,m$.
The variational problem consists of finding $v(\bm x,t) \in V^h$ and $v_k(\bm x,t) \in V^h, \ k = 1,2,\ldots,m$, satisfying:
\begin{equation}\label{4.1}
 \left ( \widetilde{ K}^m \nabla v, \nabla z \right )
 + \sum_{k=1}^m \left ( D^k \nabla v_k, \nabla z \right ) 
 = (\bm g, \nabla z) ,
\end{equation} 
\begin{equation}\label{4.2}
	\left(\frac{\partial v_k}{\partial t}, z_k\right) + \lambda_k (v_k, z_k) - (v, z_k) = 0,
	\quad k = 1,2,\ldots,m, 
	\quad t > 0,
\end{equation} 
for all $z(\bm x), \ z_k(\bm x), \ k = 1,2,\ldots,m \in V^h.$ 
Here $(\cdot, \cdot)$ denotes the inner product in $L_2(\Omega)$.
In \eqref{4.1}, it is taken into account that the boundary conditions \eqref{3.19} are natural for equation \eqref{3.18}.
The initial conditions \eqref{3.21} yield
\begin{equation}\label{4.3}
	(v_k(x,0), z_k(x)) = 0,
	\quad k = 1,2,\ldots,m.
\end{equation}
Establishing an a priori stability estimate for the solution of problem \eqref{4.1}--\eqref{4.3} of the type \eqref{3.22} is difficult. This is because the finite element solution does not have the necessary smoothness to carry out an analysis similar to the proof of Theorem~\ref{t-1}.
We note that such problems do not arise when using finite volume spatial approximations.

One can use a more complex variant instead of \eqref{4.2}.
We introduce functions $r_k(\bm x,t)$ by the expression 
\[
 r_k = \frac{\partial c_k}{\partial t} + \lambda_k c_k - \widetilde{p} ,
 \quad k = 1,2,\ldots,m.
\]
We will determine them not as solutions of equations (cf. \eqref{3.20})
\[
 r_k(\bm x,t) = 0,
 \quad k = 1,2,\ldots,m ,
\]
but as solutions of boundary value problems:
\begin{equation}\label{4.4}
- \nabla \cdot \left ( \widetilde{D}^k \nabla r_k \right ) = 0,
\quad \bm x \in \Omega,
\end{equation}
\begin{equation}\label{4.5}
\left ( \widetilde{D}^k \nabla r_k \right ) \cdot \bm \nu = 0,
\quad \bm x \in \partial \Omega,
	\quad
\end{equation}
for functions $(r_k, 1) = 0, \  k = 1,2,\ldots,m.$ 
With this approach, instead of \eqref{4.2}, the finite element solution is determined from
\begin{equation}\label{4.6}
	\left(\widetilde{D}^k \nabla \frac{\partial v_k}{\partial t}, \nabla z_k\right ) + \lambda_k \left (\widetilde{D}^k \nabla v_k, \nabla z_k \right ) - \left (\widetilde{D}^k \nabla v, \nabla z_k \right ) = 0,
	\quad k = 1,2,\ldots,m, 
	\quad t > 0,
\end{equation}  
taking into account the boundary condition \eqref{4.5}.

For time approximation, for simplicity, we will use a uniform time grid with step $\tau$: $t_n = n\tau$, $n = 0,1,\ldots$. The approximate solution of problem \eqref{4.1}--\eqref{4.3} at time $t_n$ is denoted by $v^n$, $v_k^n$, $k = 1,2,\ldots,m$.
We will use standard two-level schemes with a weight parameter $\sigma$ \cite{SamarskiiTheory,SamarskiiMatusVabischevich2002}:
\begin{equation}\label{4.7}
 \left ( \widetilde{ K}^m \nabla v^{n+\sigma}, \nabla z \right )
 + \sum_{k=1}^m \left ( D^k \nabla v_k^{n+\sigma}, \nabla z \right ) 
 = (\bm g^{n+\sigma}, \nabla z) ,
\end{equation} 
\begin{equation}\label{4.8}
\begin{split}
 \left(\widetilde{D}^k \nabla \frac{v_k^{n+1} - v_k^{n}}{\tau}, \nabla z_k\right)  & + \lambda_k \left (\widetilde{D}^k \nabla v_k^{n+\sigma}, \nabla z_k \right ) - \left (\widetilde{D}^k \nabla v^{n+\sigma}, \nabla z_k \right ) = 0 , 
 \quad k  = 1,2,\ldots,m, 
\end{split}
\end{equation} 
using the notation
\[
	v^{n+\sigma} = \sigma v^{n+1} + (1-\sigma) v^{n} ,
    \quad n = 0,1,\ldots.
\]
The initial state is determined from
\begin{equation}\label{4.9}
 \left ( \widetilde{ K}^m \nabla v^{0}, \nabla z \right )
 = (\bm g^{0}, \nabla z) ,
\end{equation}
\begin{equation}\label{4.10}
	(v_k^0, z_k) = 0,
	\quad k = 1,2,\ldots,m .
\end{equation}
For $\sigma = 1/2$, the difference scheme \eqref{4.7}--\eqref{4.9} has second-order approximation in $\tau$, and for $\sigma \neq 1/2$, first-order approximation. Stability is established by analogy with Theorem~\ref{t-1}.

\begin{thm}\label{t-2}
The difference scheme \eqref{4.7}--\eqref{4.10} is unconditionally stable for $\widetilde{ K}^m > 0$ if the weight parameter $\sigma \geq 1/2$. Under these conditions, the solution satisfies the a priori estimate ($T = N \tau$):
\begin{equation}\label{4.11}
\tau \sum_{n=0}^{N-1} \left ( \widetilde{ K}^m \nabla v^{n+\sigma}, \nabla v^{n+\sigma} \right )
+ \sum_{k=1}^m \left ( D^k \nabla v_k^{N}, \nabla v_k^{N} \right ) 
 \le \tau \sum_{n=0}^{N-1} \left ( \left ( \widetilde{ K}^m  \right )^{-1} \nabla \bm g^{n+\sigma}, \nabla \bm g^{n+\sigma} \right ) .
\end{equation} 
\end{thm}

\begin{proof}
Analogously to \eqref{3.23}, from \eqref{4.8} we have
\begin{equation}\label{4.12}
\left (D^k \nabla v_k^{n+\sigma}, \nabla z_k \right )  = \left (\widetilde{D}^k \nabla v^{n+\sigma}, \nabla z_k \right ) - 
 \left(\widetilde{D}^k \nabla \frac{v_k^{n+1} - v_k^{n}}{\tau}, \nabla z_k\right) ,
 \quad k = 1,2,\ldots,m .
\end{equation} 
Substituting into \eqref{4.7} and setting $z = v^{n+\sigma}$, taking into account the introduced notation we obtain 
\begin{equation}\label{4.13}
 \left ( \widetilde{ K}^m \nabla v^{n+\sigma}, \nabla v^{n+\sigma} \right )
 + \sum_{k=1}^m \left ( \widetilde{D}^k \nabla v^{n+\sigma}, \nabla v^{n+\sigma} \right ) 
 - \sum_{k=1}^m \left ( \widetilde{D}^k \nabla \frac{v_k^{n+1} - v_k^{n}}{\tau}, \nabla v^{n+\sigma} \right ) 
 = (\bm g^{n+\sigma}, \nabla v^{n+\sigma}) .
\end{equation} 
In \eqref{4.12} we set
\[
  z_k = \frac{v_k^{n+1} - v_k^{n}}{\tau} ,
  \quad k = 1,2,\ldots,m ,
\]
which gives
\begin{equation}\label{4.14}
\begin{split}
\left (D^k \nabla v_k^{n+\sigma}, \nabla \frac{v_k^{n+1} - v_k^{n}}{\tau} \right )  & - \left (\widetilde{D}^k \nabla v^{n+\sigma}, \nabla \frac{v_k^{n+1} - v_k^{n}}{\tau} \right ) \\  
& + 
 \left(\widetilde{D}^k \nabla \frac{v_k^{n+1} - v_k^{n}}{\tau}, \nabla \frac{v_k^{n+1} - v_k^{n}}{\tau} \right) = 0, 
 \quad k  = 1,2,\ldots,m .
\end{split}
\end{equation} 
Adding \eqref{4.13} and \eqref{4.14}, we obtain the inequality
\begin{equation}\label{4.15}
 \left ( \widetilde{ K}^m \nabla v^{n+\sigma}, \nabla v^{n+\sigma} \right )
 + \sum_{k=1}^m  \left (D^k \nabla v_k^{n+\sigma}, \nabla \frac{v_k^{n+1} - v_k^{n}}{\tau} \right ) 
 \leq (\bm g^{n+\sigma}, \nabla v^{n+\sigma}) .
\end{equation} 
Taking into account that
\[
v_k^{n+\sigma} = \left(\sigma - \frac{1}{2}\right) \tau \frac{v_k^{n+1} - v_k^{n}}{\tau} + \frac{v_k^{n+1} + v_k^{n}}{2},
\]
for the second term in \eqref{4.15} when $\sigma \geq 1/2$ we have
\[
 \left (D^k \nabla v_k^{n+\sigma}, \nabla \frac{v_k^{n+1} - v_k^{n}}{\tau} \right ) \le 
 \frac{1}{2 \tau} \left ( \left ( D^k \nabla v_k^{n+1}, \nabla v_k^{n+1}  \right ) -
\left ( D^k \nabla v_k^{n}, \nabla v_k^{n}  \right ) \right ) .
\]
Substituting into \eqref{4.15} and taking into account
\[
 (\bm g^{n+\sigma}, \nabla v^{n+\sigma}) \le \frac 12 
  \left ( \widetilde{ K}^m \nabla v^{n+\sigma}, \nabla v^{n+\sigma} \right )
  + \frac 12 \left ( \left ( \widetilde{ K}^m  \right )^{-1} \nabla \bm g^{n+\sigma}, \nabla \bm g^{n+\sigma} \right ) ,
\]
we obtain the inequality
\[
\begin{split}
 \tau \left ( \left ( \widetilde{ K}^m  \right )^{-1} \nabla \bm g^{n+\sigma}, \nabla \bm g^{n+\sigma} \right ) 
 & + \sum_{k=1}^m \left ( D^k \nabla v_k^{n+1}, \nabla v_k^{n+1}  \right ) - \sum_{k=1}^m 
 \left ( D^k \nabla v_k^{n}, \nabla v_k^{n}  \right ) \\
 & \le \tau \left ( \left ( \widetilde{ K}^m  \right )^{-1} \nabla \bm g^{n+\sigma}, \nabla \bm g^{n+\sigma} \right ) .
\end{split}
\]
Summing over $n$ from $0$ to $N-1$ leads us to the desired estimate \eqref{4.11}.
\end{proof}

We note some features of the computational implementation of the considered two-level schemes \eqref{4.7}--\eqref{4.10}.
We seek an approximate solution of the boundary value problem for a weakly coupled system of equations.
We would like to be able to advance to a new time level by solving separate subproblems for the unknowns.

We consider a variant of spatial approximation using the same finite element basis for both $v(\bm x,t)$ and $v_k(\bm x,t)$, $k = 1,2,\ldots,m$. This allows us to express $v_k^{n+1}$ from \eqref{4.8} as follows:
\begin{equation}\label{4.16}
 \left(D^k \nabla v_k^{n+1}, \nabla z\right) =
 \frac{\tau}{1 + \sigma \lambda_k \tau} 
 \left (D^k \nabla v^{n+\sigma}, \nabla z \right )  
 + (f_k^{n+\sigma}, z),
\end{equation} 
where 
\[
	(f_k^{n+\sigma}, z) = \frac{1}{1 + \sigma \lambda_k \tau} \Big(\big(1 - (1-\sigma) \lambda_k \tau\big) \left(D^k \nabla v_k^{n}, \nabla z\right),
	\quad k = 1,2,\ldots,m.
\]
Substituting into \eqref{4.7}, we obtain the following problem for finding $v^{n+\sigma}$:
\begin{equation}\label{4.17}
 \left ( \widetilde{ K}^m \nabla v^{n+\sigma}, \nabla z \right )
 + \sum_{k=1}^m  \frac{\tau}{1 + \sigma \lambda_k \tau} \left (D^k \nabla v^{n+\sigma}, \nabla z \right ) 
 = (\bm g^{n+\sigma}, \nabla v^{n+\sigma}) - (f_k^{n+\sigma}, z).
\end{equation} 
For $v^{n+1}$ we have
\[
 v^{n+1} = \frac{1}{\sigma} \left (v^{n+\sigma} - (1-\sigma) v^{n} \right ) .
\]  
Thus, determining $v^{n+1}$ at the new time level is carried out by solving problem \eqref{4.17}, which is similar to the standard problem of steady-state filtration with a tensor permeability coefficient.

The calculation of the auxiliary quantities $v_k^{n+1}$, $k = 1,2,\ldots,m$, according to \eqref{4.16} is not very good for computational implementation: we have problems with Neumann boundary conditions, and there are $m$ of them.
We formulated such a problem in order to prove the stability of the approximate solution as simply as possible.
Simplification of the problem of calculating the auxiliary quantities $v_k^{n+1}$, $k = 1,2,\ldots,m$ is achieved by transitioning from equations \eqref{4.6} to equations \eqref{4.2}.
In this case, instead of \eqref{4.16}, explicit calculation formulas are used:
\[
 \left( v_k^{n+1}, z\right) =
 \frac{\tau}{1 + \sigma \lambda_k \tau} 
 \left ( v^{n+\sigma}, z \right )  
 + \frac{1}{1 + \sigma \lambda_k \tau} \big(1 - (1-\sigma) \lambda_k \tau\big) \left( v_k^{n}, z\right),
	\quad k = 1,2,\ldots,m.
\]
The computational complexity of the approximate solution of the considered nonlocal problem is only slightly higher than the complexity of a standard local filtration problem. It only requires the additional solution of $m$ simple auxiliary local evolutionary problems using explicit procedures.

\section{Numerical examples}\label{sec:5}

The proposed computational technology for describing unsteady flows in a porous medium taking into account memory effects is illustrated with a two-dimensional test problem.
At the macroscale, the computational domain is a rectangle
\[
  \Omega = \{ \bm x = (x_1, x_2) \ | \ 0 < x_1 < 2, \ 0 < x_2 < 1 \} .
\]
In the absence of body forces ($\bm f(\bm x,t) = 0$) and without accounting for memory effects, steady-state filtration is described by the equation
\begin{equation}\label{5.1}
- \nabla \cdot \left ( \accentset{-}{K} \nabla p_0 \right )
= 0, 
\quad \bm x \in \Omega .
\end{equation} 
Equation \eqref{5.1} is supplemented with the boundary conditions
\begin{equation}\label{5.2}
 p_0(0,x_2) = 0,
 \quad p_0(2,x_2) = 1,
 \quad \left ( \accentset{-}{K} \nabla p_0 \right ) \cdot \bm \nu = 0, 
 \quad x_2 = 0, 1 .
\end{equation} 

The permeability tensor $\accentset{-}{K}$ is determined according to \eqref{3.7}--\eqref{3.10}. We consider the case where the solid part of the periodicity cell $Y_s$ is an ellipse rotated by 45 degrees, its area is $\pi / 12$, and the ratio of the ellipse's semi-axes is $\gamma$. The calculations were performed on a mesh with a characteristic cell size $h = 0.01$; quadratic finite elements were used for velocity and linear elements for pressure.

\begin{figure}[ht]
	\centering
	\includegraphics[width=0.48\linewidth]{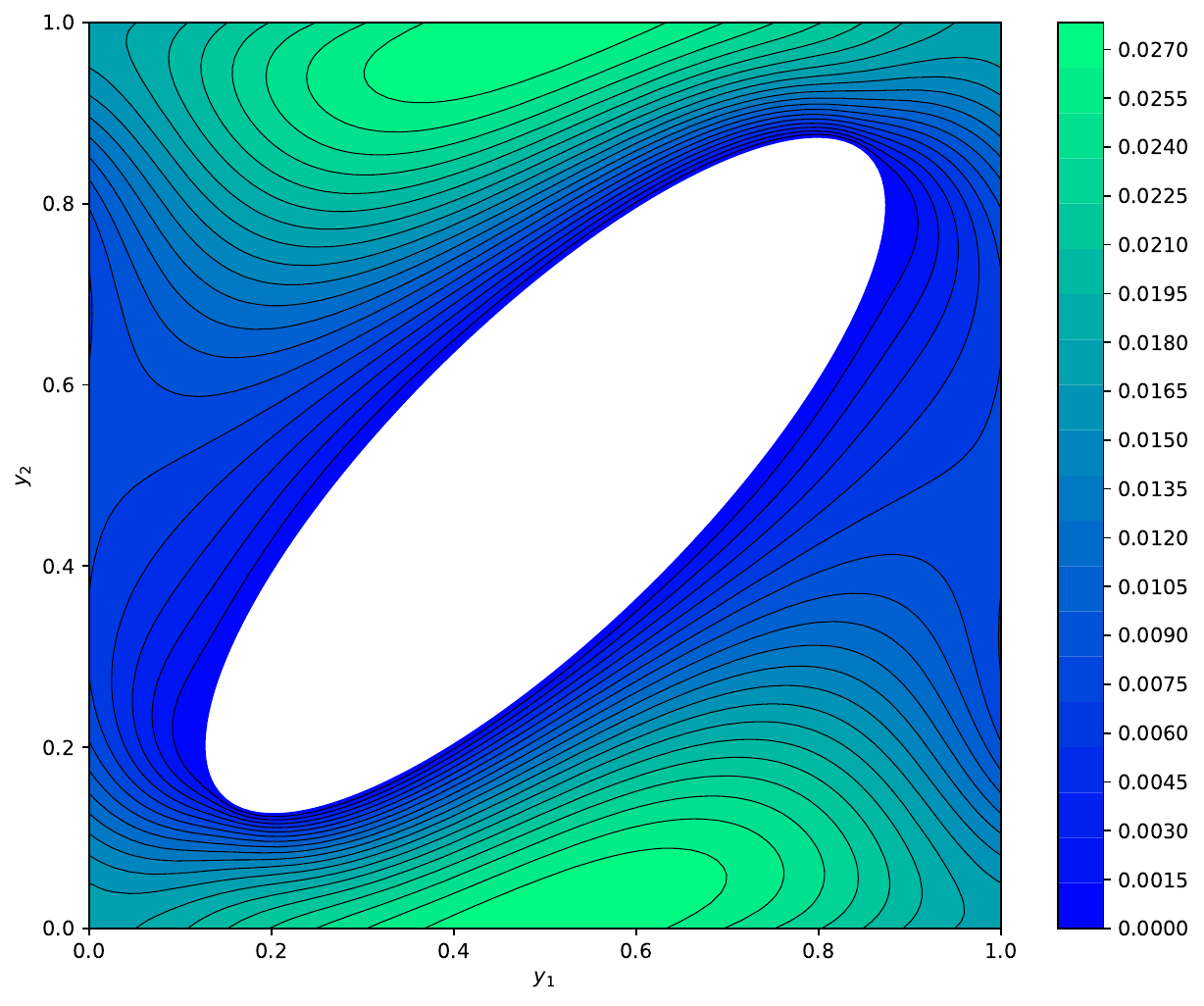}
	\includegraphics[width=0.48\linewidth]{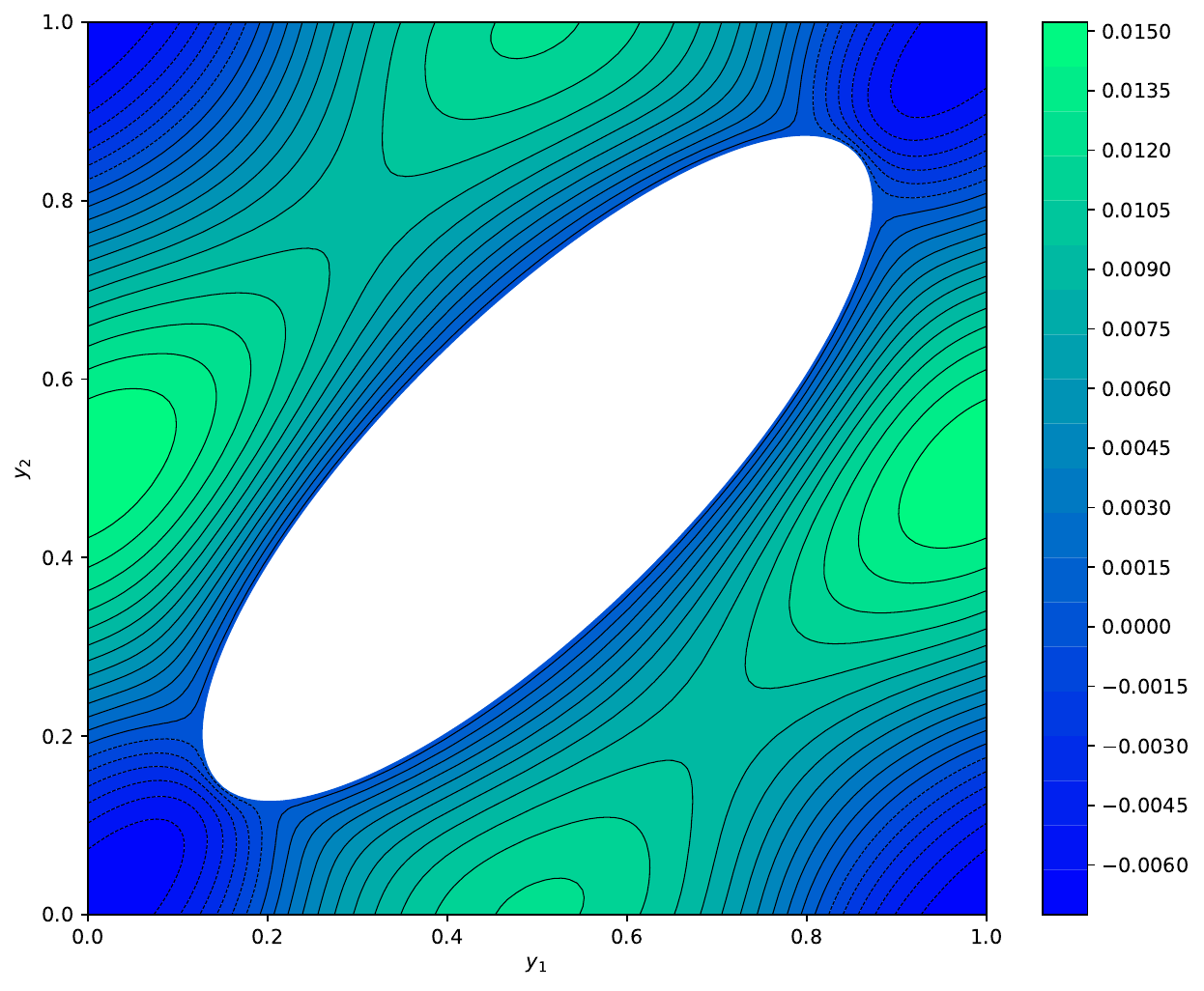}
	\caption{Auxiliary functions $\bm w_1^0 \cdot \bm e_1$ (left) and $\bm w_2^0 \cdot \bm e_1$ (right).}
	\label{f-2}
\end{figure}

Fig.~\ref{f-2} shows, as an example, the components of the solution of the cell problem \eqref{3.7}--\eqref{3.9} for $\gamma = 3$. After processing these solutions according to \eqref{3.10}, we calculate the permeability tensor. The corresponding data for different aspect ratios of the ellipse are given in Table~\ref{tab-1}. As the elongation of the ellipse increases, the off-diagonal components of the permeability tensor increase, while the diagonal components decrease.

\begin{table}[ht]
\centering
\caption{Permeability tensor}
\label{tab-1}
\vspace{1ex}
\begin{tabular}{lcc}
\toprule
$\gamma$ & $\accentset{-}{K}_{11} = \accentset{-}{K}_{22}$ & $\accentset{-}{K}_{12} = \accentset{-}{K}_{21}$ \\
\midrule    
1 & 0.01269975 & 0.00000000 \\
2 & 0.01144540 & 0.00251806 \\
3 & 0.00981454 & 0.00437231 \\
4 & 0.00855774 & 0.00604958 \\
\bottomrule
\end{tabular}
\end{table}

The calculated permeability tensors are used to solve the filtration problem \eqref{5.1}, \eqref{5.2} at the macroscale.  
For $\gamma = 1$ (scalar permeability coefficient), we obtain a linear pressure profile in the horizontal direction.
Anisotropic permeability (see Fig.~\ref{f-3}) leads to a significant nonuniformity of pressure across the height.

\begin{figure}[ht]
	\centering
	\includegraphics[width=0.49\linewidth]{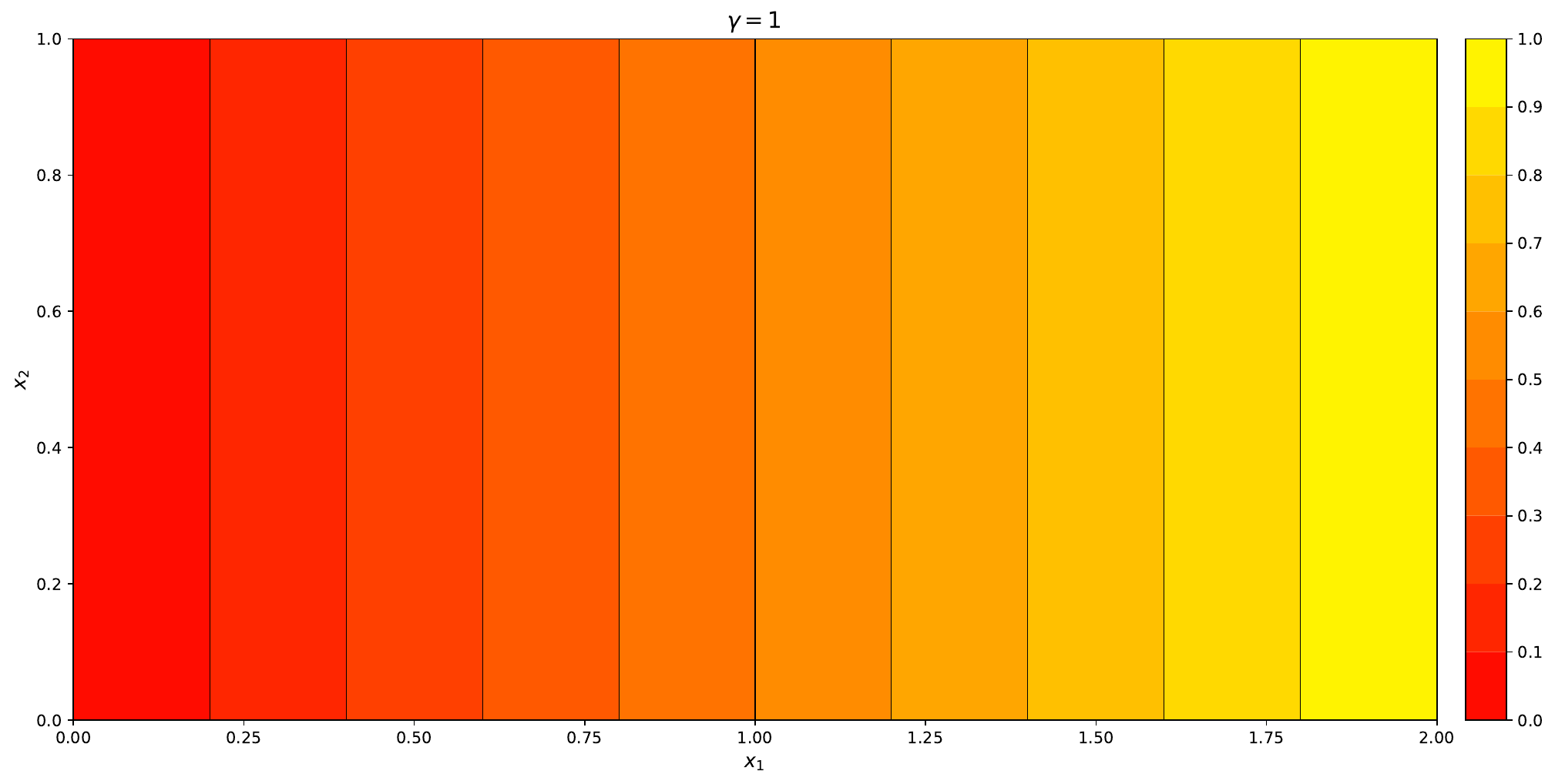} 
	\includegraphics[width=0.49\linewidth]{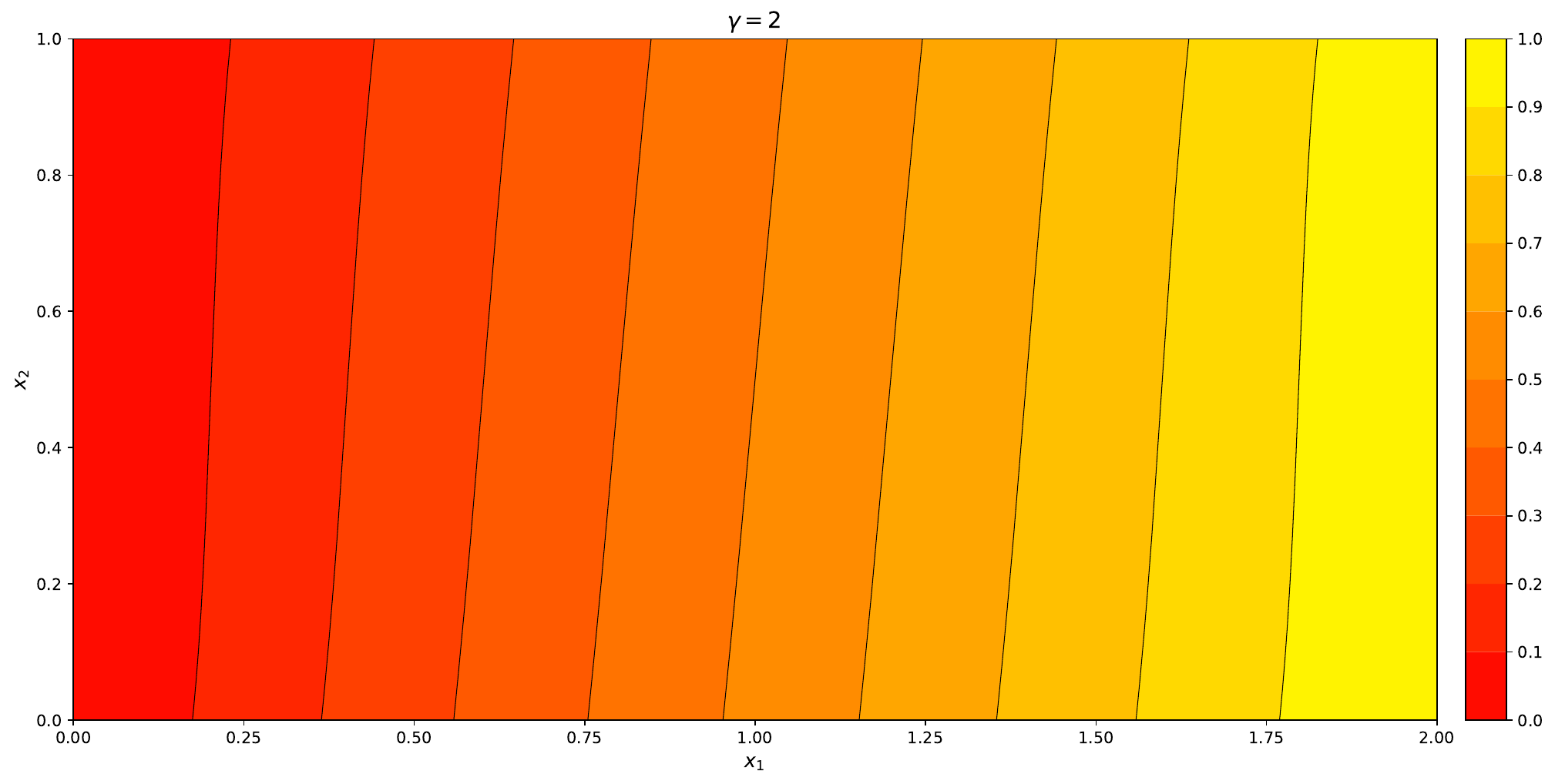} \\
	\includegraphics[width=0.49\linewidth]{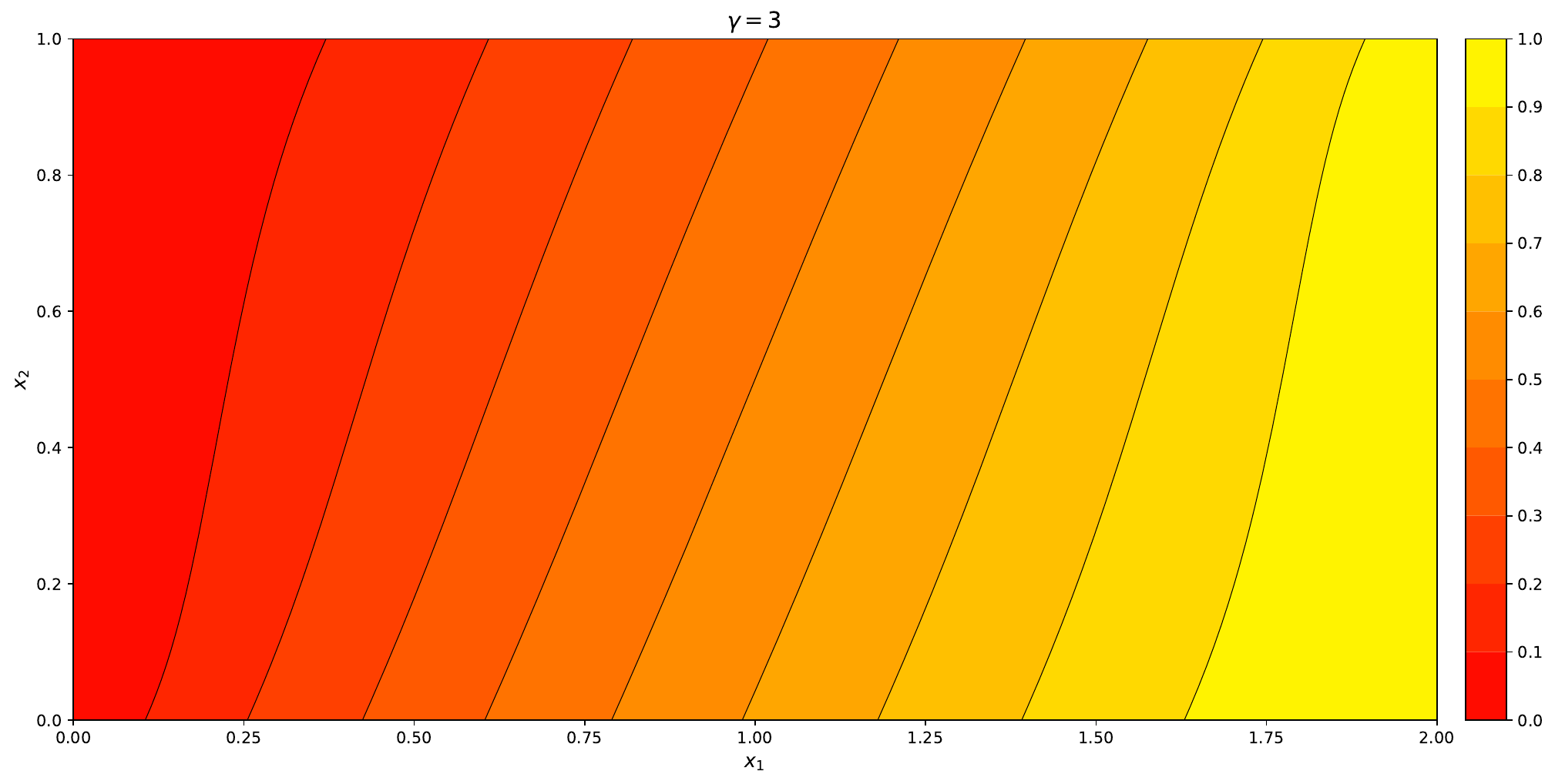} 
	\includegraphics[width=0.49\linewidth]{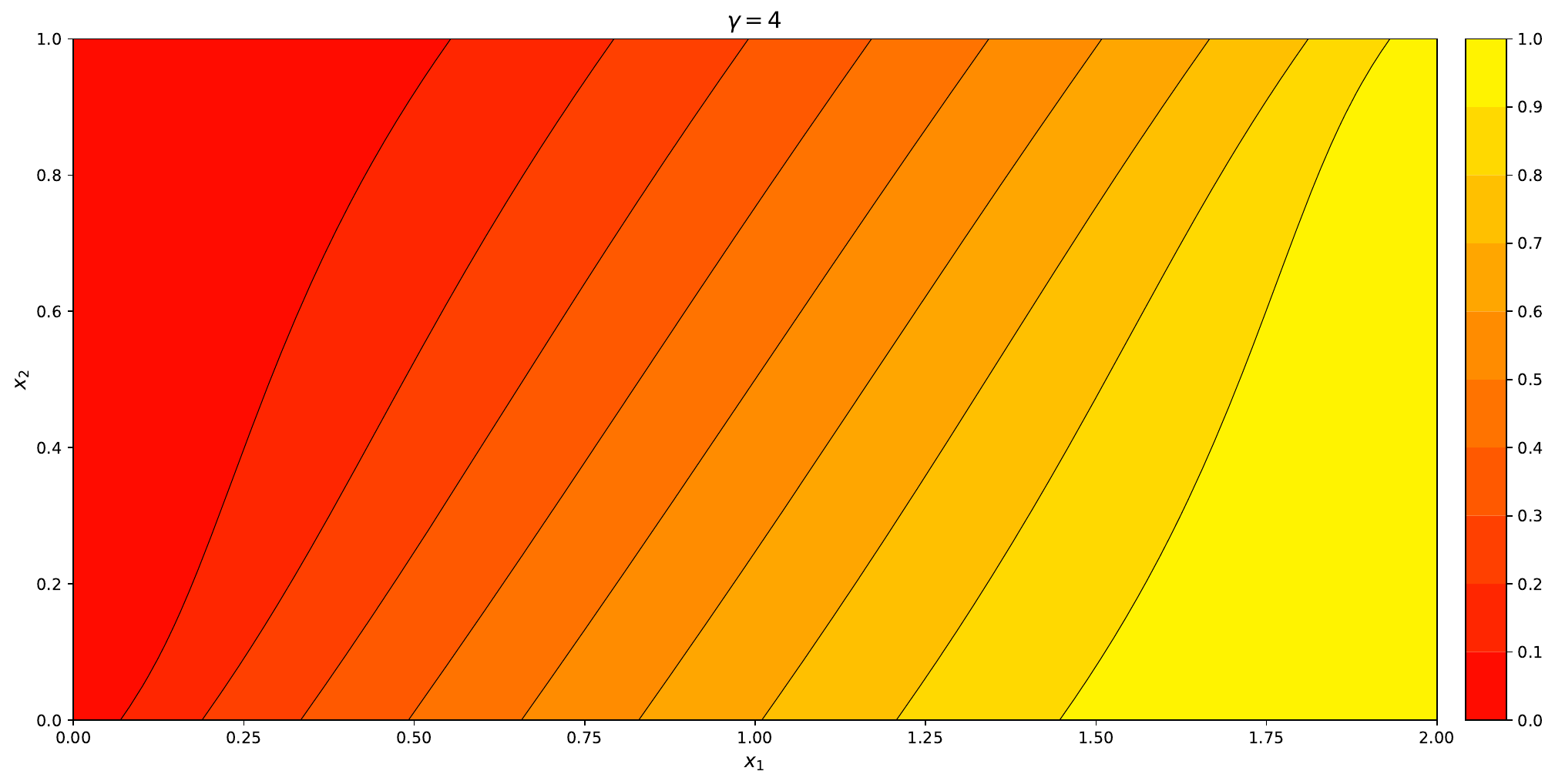} \\    
	\caption{Pressure for different pore geometries.}
	\label{f-3}
\end{figure}

The approximation of the memory kernel is performed by solving the spectral problem \eqref{3.1}--\eqref{3.3} in the domain $Y_f$.
The accuracy of the eigenvalue computation is controlled by calculations on a sequence of refined meshes.
We use three computational meshes: Mesh 1 --- characteristic mesh size $h = 0.02$, 2,365 nodes; Mesh 2 --- $h = 0.01$, 8,973 nodes; Mesh 3 --- $h = 0.005$, 35,015 nodes.
The first 10 eigenvalues are presented in Table~\ref{tab-2}. We observe convergence of the eigenvalues as the mesh is refined. The components $\bm \varphi_k \cdot \bm e_1, \ k = 1,2, \ldots$ of the first four eigenfunctions obtained on the primary computational mesh 2 are shown in Fig.~\ref{f-4}. Due to the symmetry of the solution of the considered cell problem, the components $\bm \varphi_k \cdot \bm e_2, \ k = 1,2, \ldots$ are not presented.
The computed eigenfunctions $\bm \varphi_k, \ k = 1,2, \ldots$ are normalized: $\|\bm \varphi_k \| = 1, \ k = 1,2, \ldots$.

\begin{table}[ht]
\centering
\caption{First 10 eigenvalues}
\label{tab-2}
\vspace{1ex}
\begin{tabular}{rrrr}
\toprule
$k$ & Mesh 1 & Mesh 2 & Mesh 3 \\
\midrule    
1 & 40.33104  &     40.35215  &    40.35746 \\
2 & 51.14206  &     51.23001  &    51.25329 \\
3 & 114.24218  &     114.35255  &    114.38035 \\
4 & 139.04402  &     139.18545  &    139.22217 \\
5 & 165.53322  &     165.60993  &    165.62792 \\
6 & 171.49287  &     171.72568  &    171.78598 \\
7 & 176.64171  &     176.71223  &    176.72980 \\
8 & 216.34115  &     216.66890  &    216.75484 \\
9 & 219.82942  &     219.91384  &    219.93342 \\
10 &238.26248  &     238.36223  &    238.38799 \\
\bottomrule
\end{tabular}
\end{table}

\begin{figure}[ht]
	\centering
	\includegraphics[width=0.48\linewidth]{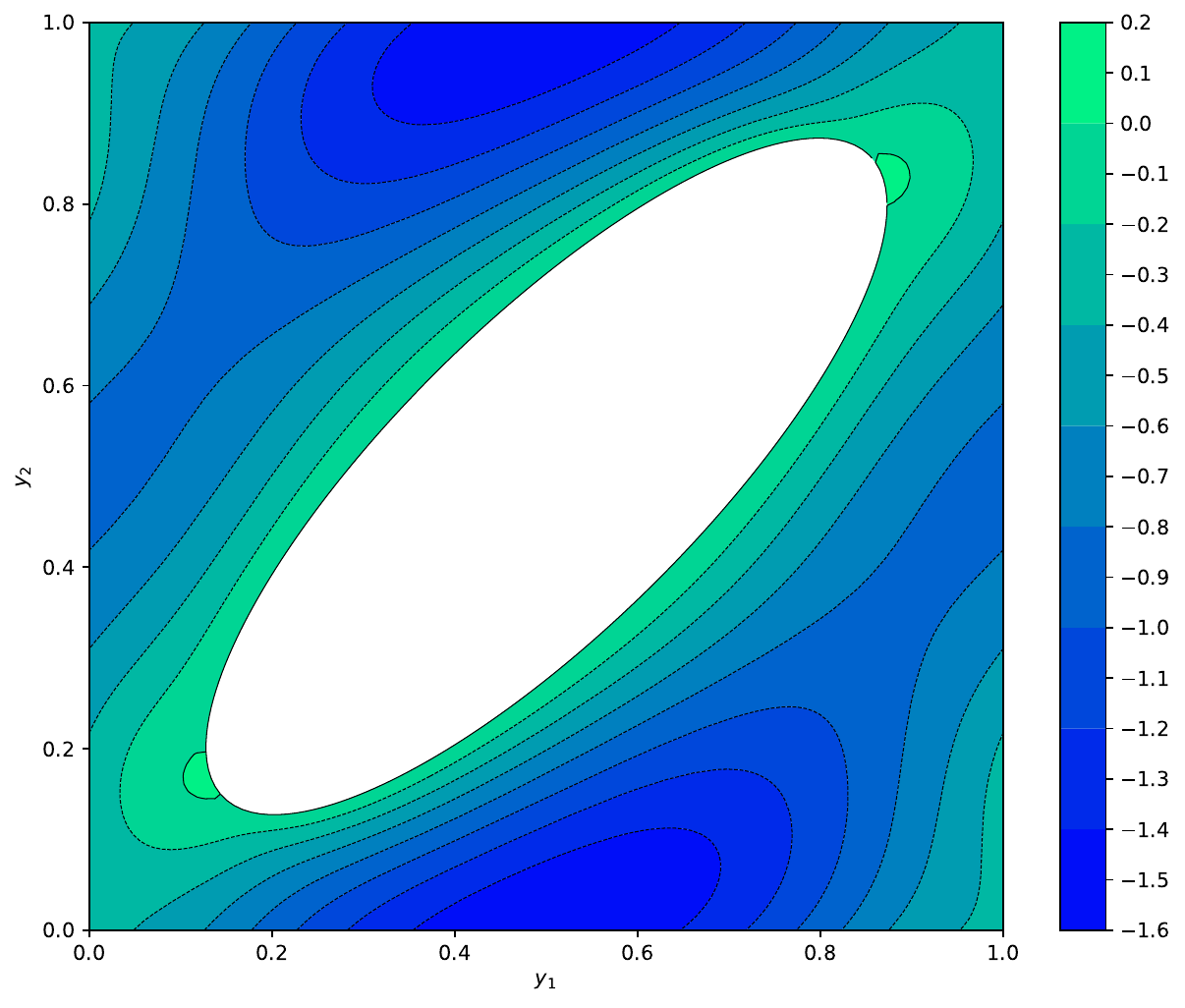}
	\includegraphics[width=0.48\linewidth]{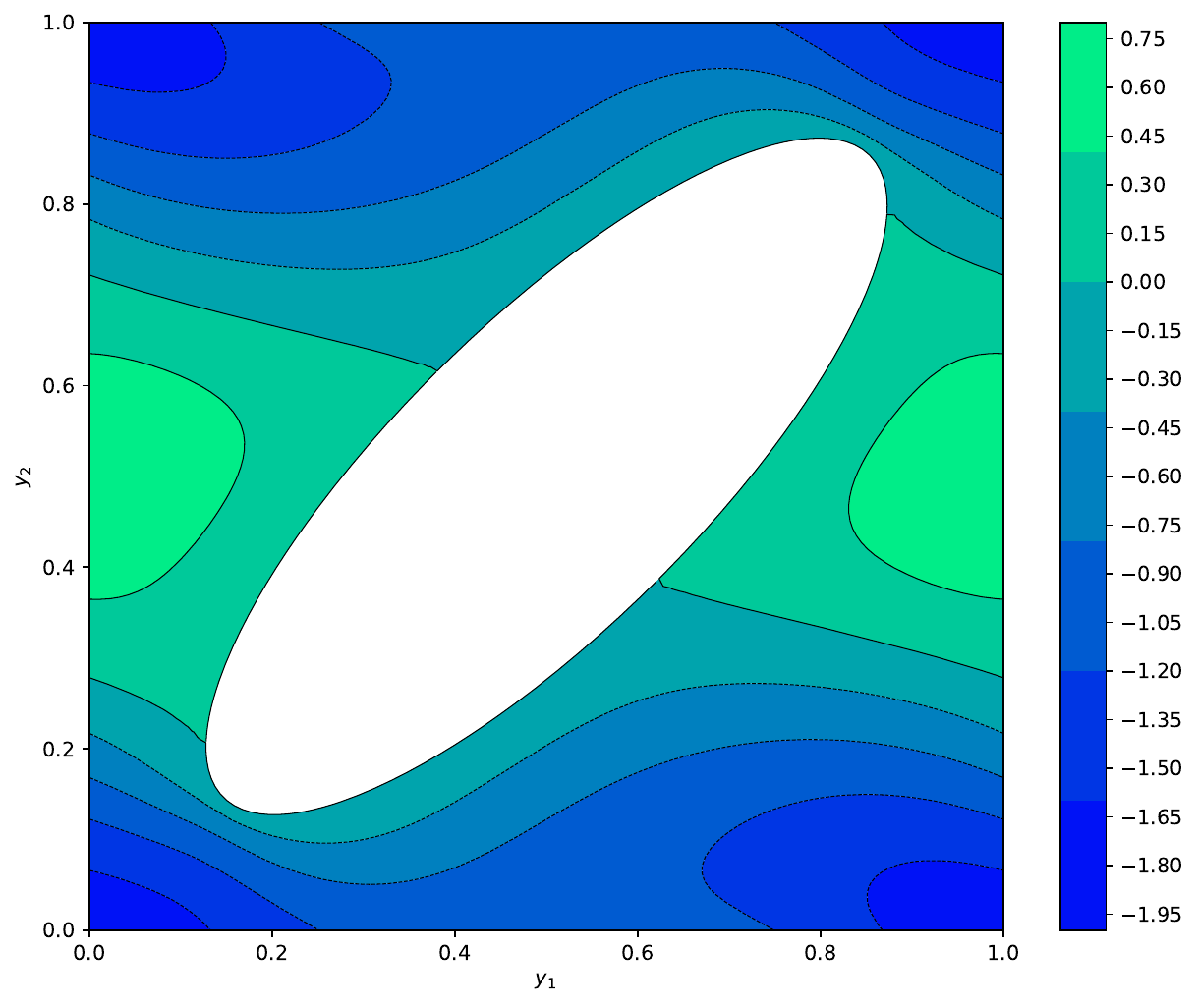} \\
	\includegraphics[width=0.48\linewidth]{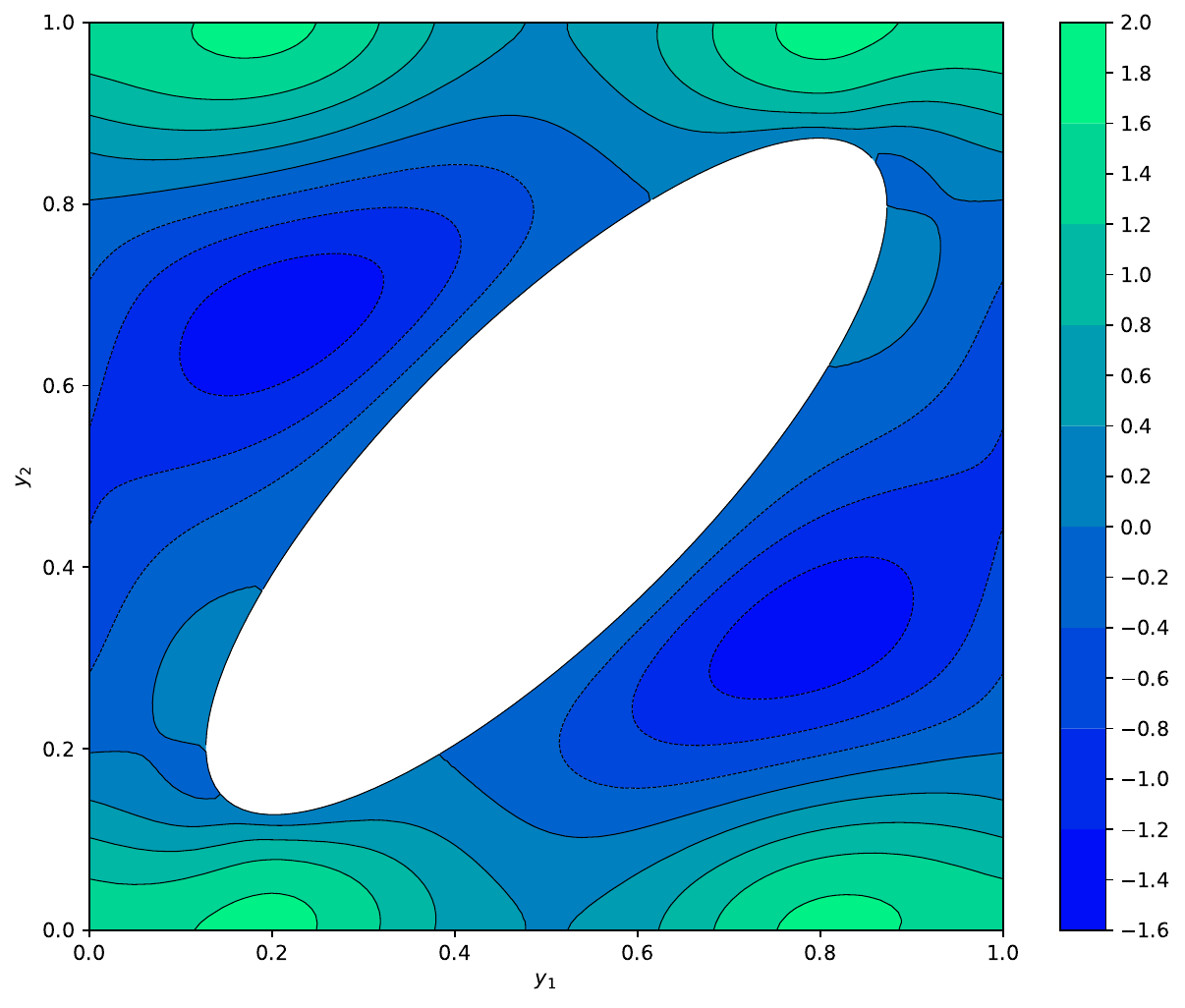}
	\includegraphics[width=0.48\linewidth]{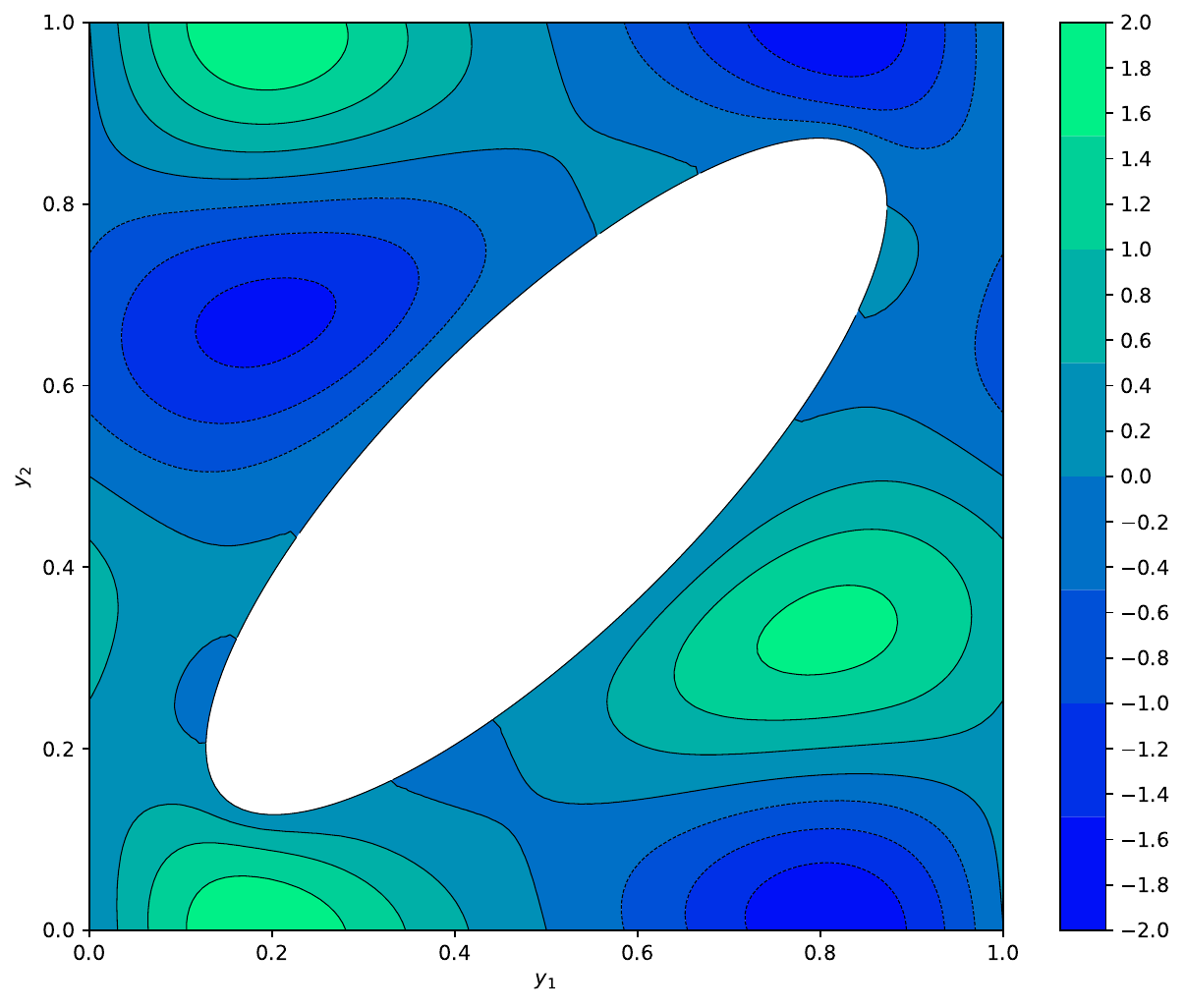} \\
	\caption{Eigenfunctions: top row --- $\bm \varphi_1 \cdot \bm e_1$ (left) and $\bm \varphi_2 \cdot \bm e_1$ (right), bottom row --- $\bm \varphi_3 \cdot \bm e_1$ (left) and $\bm \varphi_4 \cdot \bm e_1$ (right).}
	\label{f-4}
\end{figure}

To approximate the kernel, we use the first 100 eigenvalues ($m=100$ in \eqref{3.12}). The corresponding eigenvalues $\lambda_k, \ k=1,2,\dots,m$ are shown in Fig.~\ref{f-5}.
In the range $k=1,2,\dots,m$, approximately linear growth of the eigenvalues $\lambda_k$ with increasing $k$ is observed.

The components of the permeability tensor are determined (see \eqref{3.6}) by the coefficients $a_1^k, \ a_2^k, \ k=1,2,\dots.$ The first 100 such coefficients are shown in Fig.~\ref{f-6}. 
Blue color corresponds to positive values of the coefficients, and red to negative ones.
We see that a significant number of these coefficients have small amplitude.

\begin{figure}[ht]
	\centering
	\includegraphics[width=0.75\linewidth]{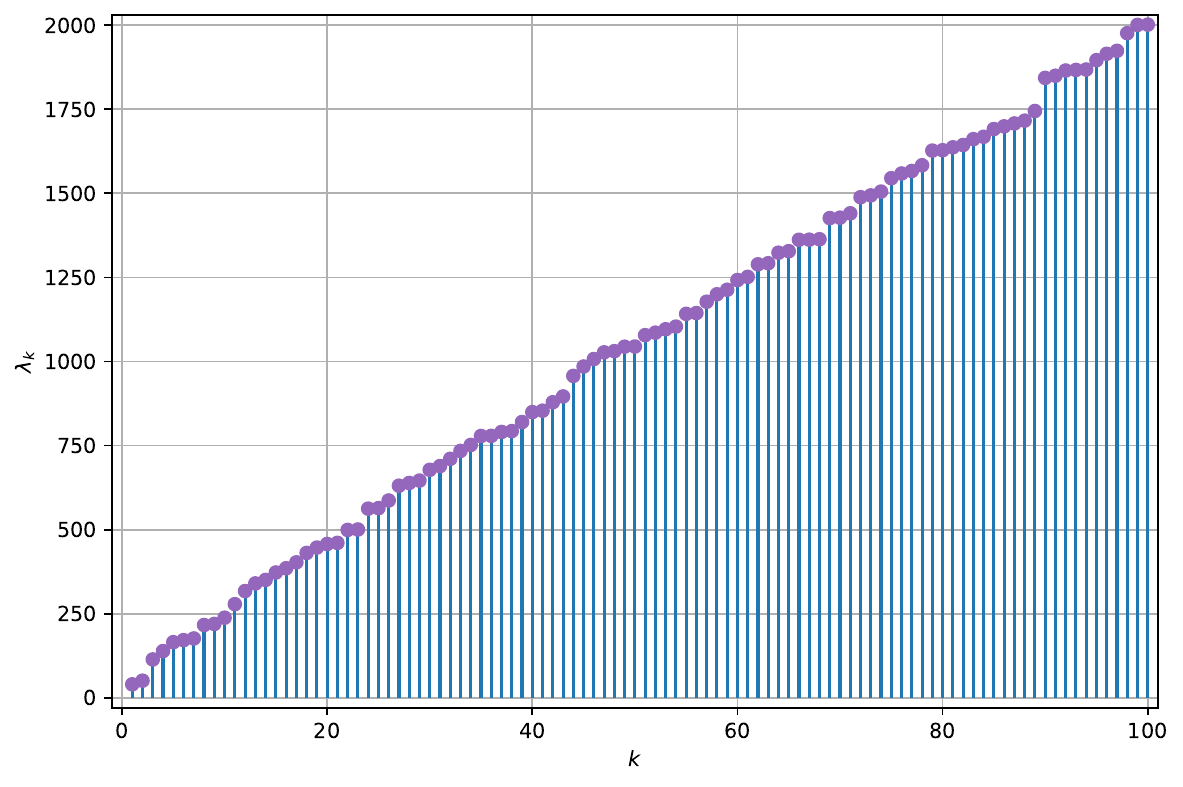}
	\caption{Eigenvalues $\lambda_k, \ k=1,2,\dots,m$.}
	\label{f-5}
\end{figure}

\begin{figure}[ht]
	\centering
	\includegraphics[width=0.75\linewidth]{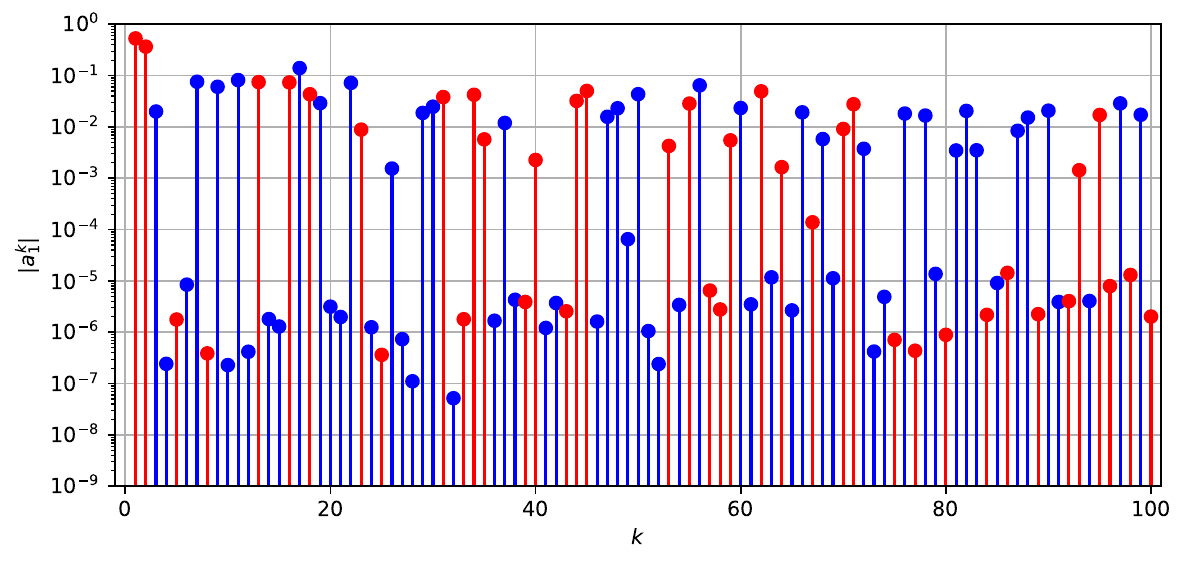}
	\includegraphics[width=0.75\linewidth]{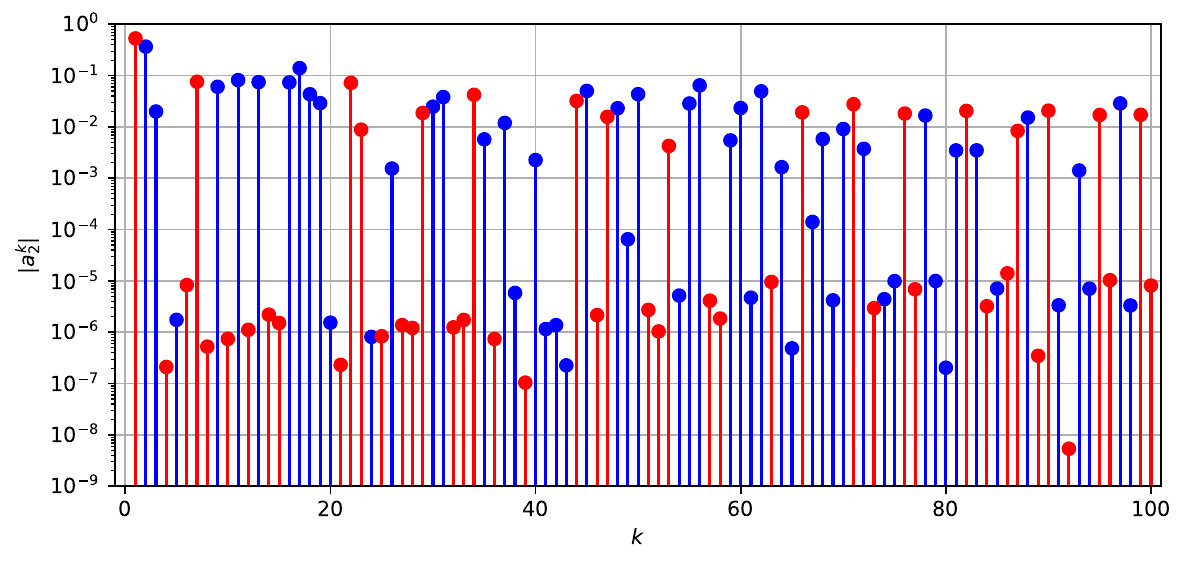}    
	\caption{Coefficients $|a_1^k|$ (top) and $|a_2^k|, \ k=1,2,\dots,m$ (bottom).}
	\label{f-6}
\end{figure}

In the macroscale problem \eqref{3.18}--\eqref{3.21}, the account of memory effects is provided by the tensors $D^k, \ k=1,2,\ldots,m$. Their influence can be estimated by the change in the steady-state permeability tensor $\widetilde{ K}^m$ without memory. The elements of this tensor are calculated according to \eqref{3.13}. For the considered test problem, the components of the tensor are shown in Fig.~\ref{f-7}. We can note a rapid decrease in the amplitude of the diagonal and off-diagonal elements for the first eigenvalues and small refinements for large $m$.

\begin{figure}[ht]
	\centering
	\includegraphics[width=0.75\linewidth]{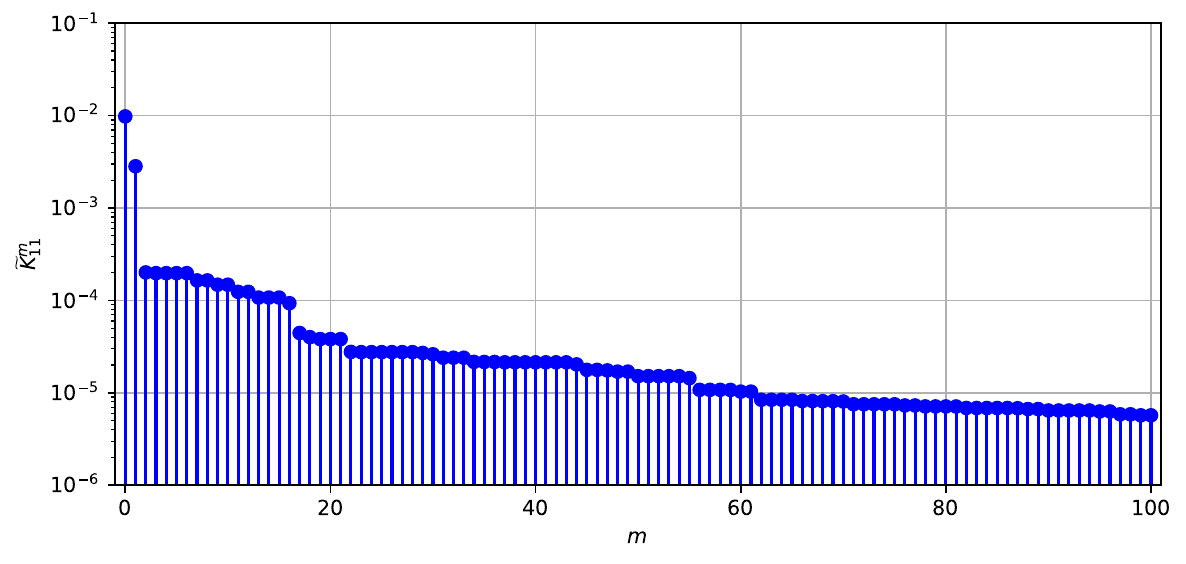}
	\includegraphics[width=0.75\linewidth]{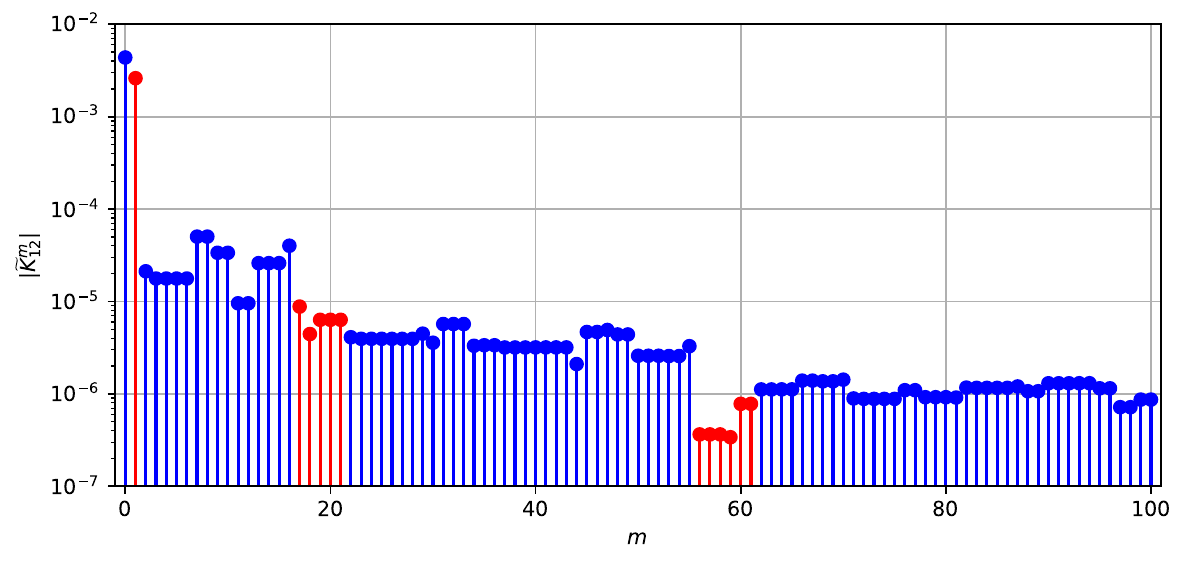}    
	\caption{Elements of the steady-state conductivity tensor: diagonal (top) and off-diagonal (bottom).}
	\label{f-7}
\end{figure}

We can neglect the influence of harmonics for which the elements of the tensor $\widetilde{D}^k$ are small:
\[
  | \widetilde{D}^k_{ij} | =\frac{|a_i^k a_j^k|}{\lambda_k} \le \epsilon .
\]
We discard those terms that exceed the threshold value $\epsilon$.
The result of such a selection for different values of $\epsilon$ is shown in Fig.~\ref{f-8}.
For example, for $\epsilon = 10^{-5}$ we take into account only 10 eigenvalues. Without a noticeable increase in the accuracy of the computation of the conductivity tensor elements, this result is achieved without filtering when accounting for already 22 eigenvalues.  
For $\epsilon = 10^{-6}$, we can limit ourselves to only 20 eigenvalues out of 62.

\begin{figure}[ht]
	\centering
	\includegraphics[width=0.75\linewidth]{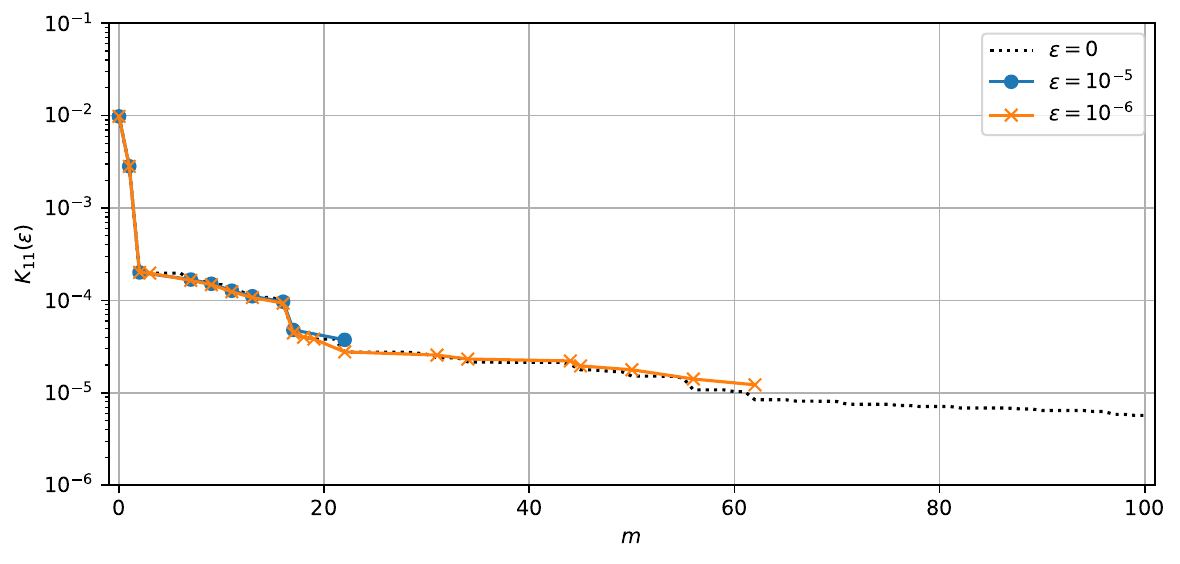}
	\includegraphics[width=0.75\linewidth]{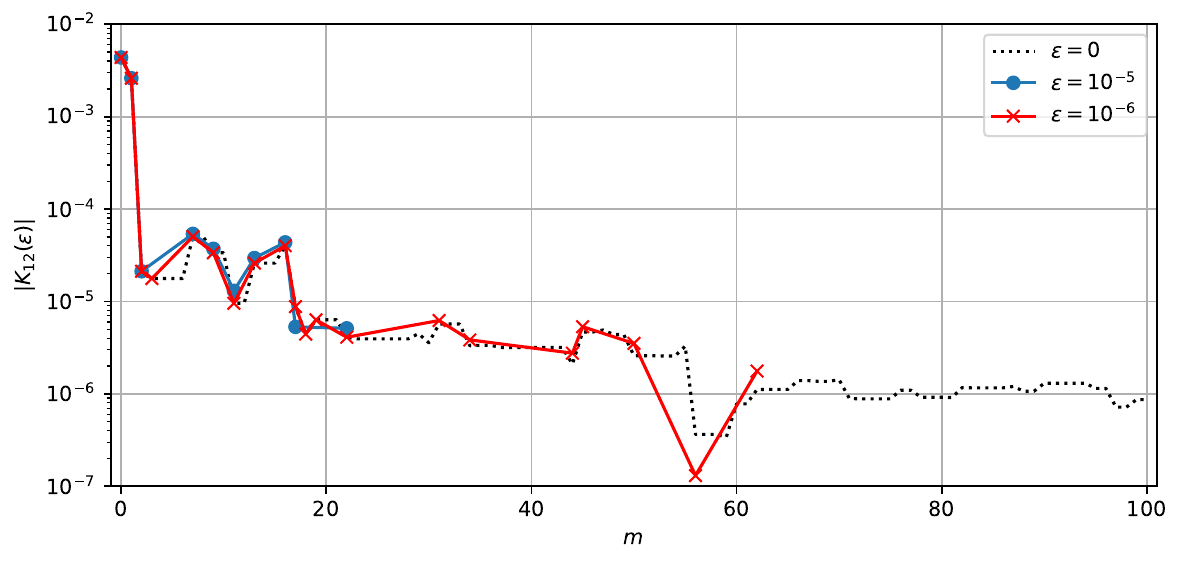}    
	\caption{Elements of the conductivity tensor $\widetilde{ K}(\epsilon)$ after filtering: diagonal (top) and off-diagonal (bottom).}
	\label{f-8}
\end{figure}

The presented computational data demonstrate that the effect of accounting for memory effects is achieved in our case already for $m = 3$. 
Taking into account a larger number of terms increases accuracy only slightly.
For $m = 3$, for the elements of the steady-state conductivity tensor $\widetilde{ K}^m$ we have
\[
 \widetilde{ K}^m_{11} = \widetilde{ K}^m_{22} = 1.97429 \cdot 10^{-4},
 \quad \widetilde{ K}^m_{12} = \widetilde{ K}^m_{21} = 1.77255 \cdot 10^{-5} .
\]
The corresponding parameters for accounting for memory effects are given in Table~\ref{tab-3}.
For this case, Fig.~\ref{f-9} shows the approximate solution of the problem at different time instants. 
The calculations were performed with a time step $\tau = 10^{-5}$. A rapid readjustment of the initial pressure profile is observed.
Upon reaching steady state, the solution of the problem taking into account memory effects coincides with the solution of the problem without memory (see Fig.~\ref{f-3} for $\gamma =3$).

\begin{table}[ht]
\centering
\caption{Parameters of the macroscale problem}
\label{tab-3}
\vspace{1ex}
\begin{tabular}{rrrr}
\toprule
 $k$ &   $\lambda_k$  & $D_{11}^k = D_{22}^k$  & $D_{12}^k = D_{21}^k$ \\
\midrule    
  1 &  40.352157 & -0.530804 & -0.530804 \\
  2 &  51.230012 & -0.367151 & 0.367151 \\
  3 &  114.352557 & 0.019996 & 0.019996 \\
\bottomrule
\end{tabular}
\end{table}

\begin{figure}[ht]
	\centering
	\includegraphics[width=0.49\linewidth]{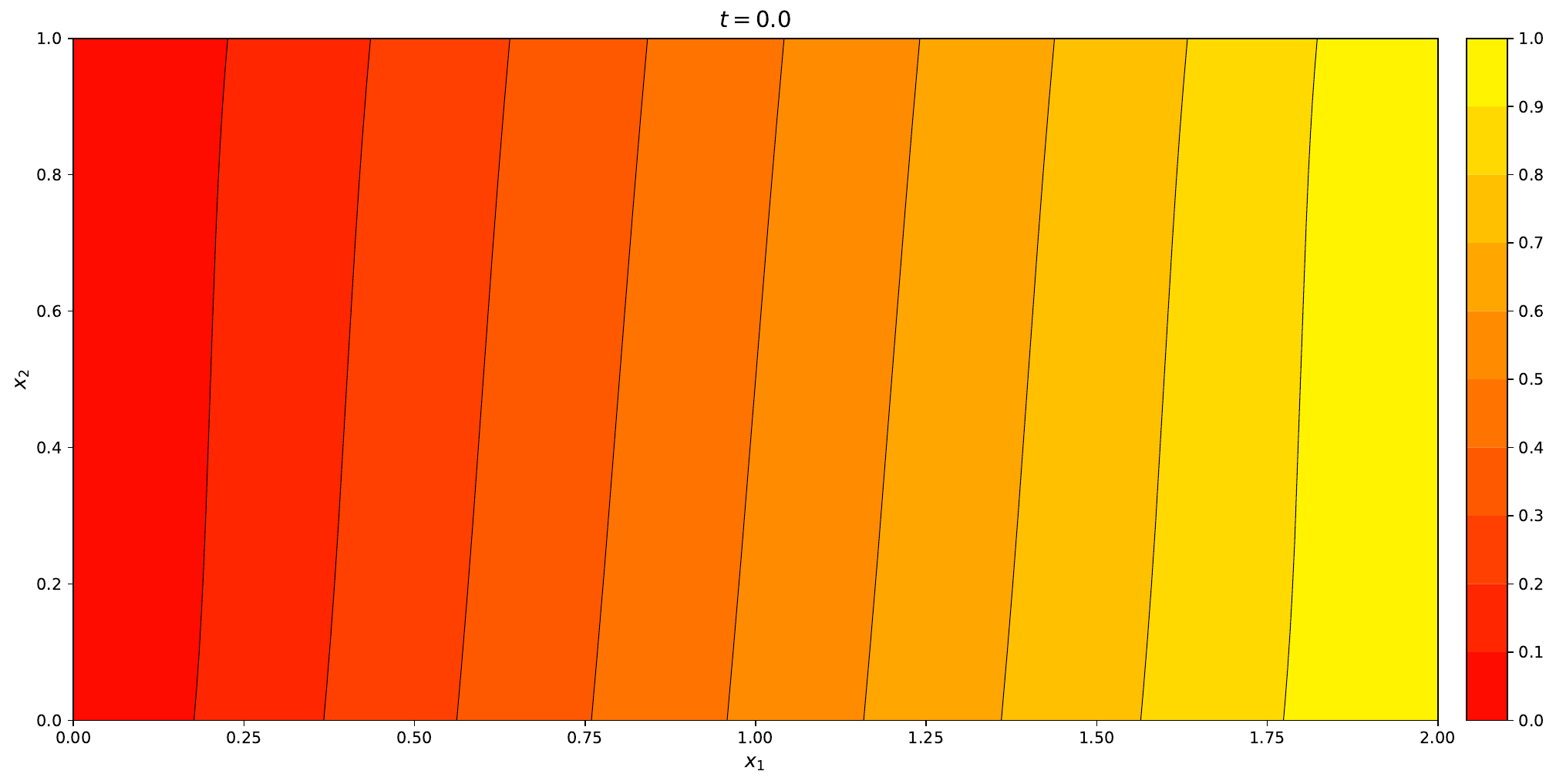}
	\includegraphics[width=0.49\linewidth]{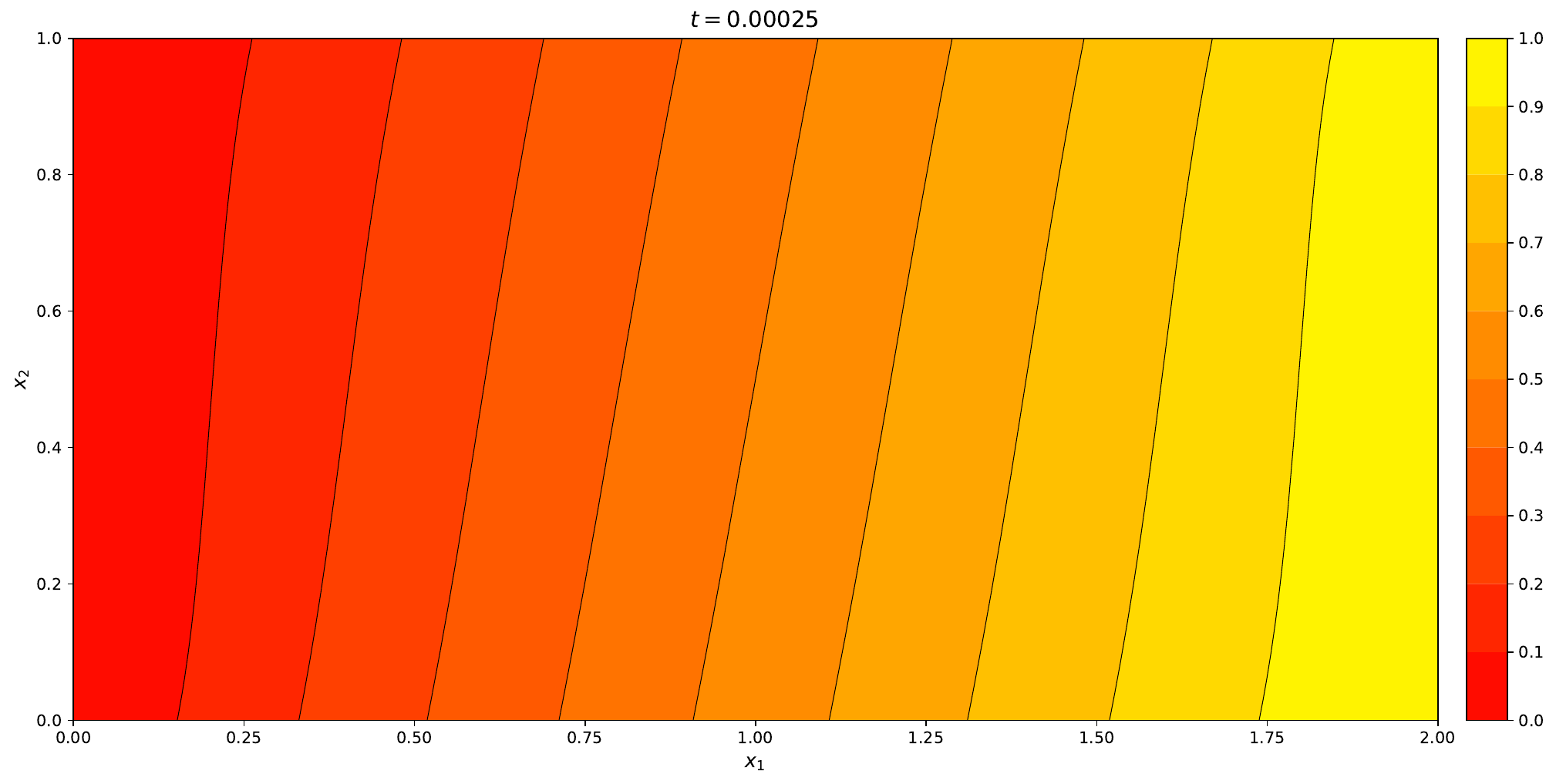}  \\
	\includegraphics[width=0.49\linewidth]{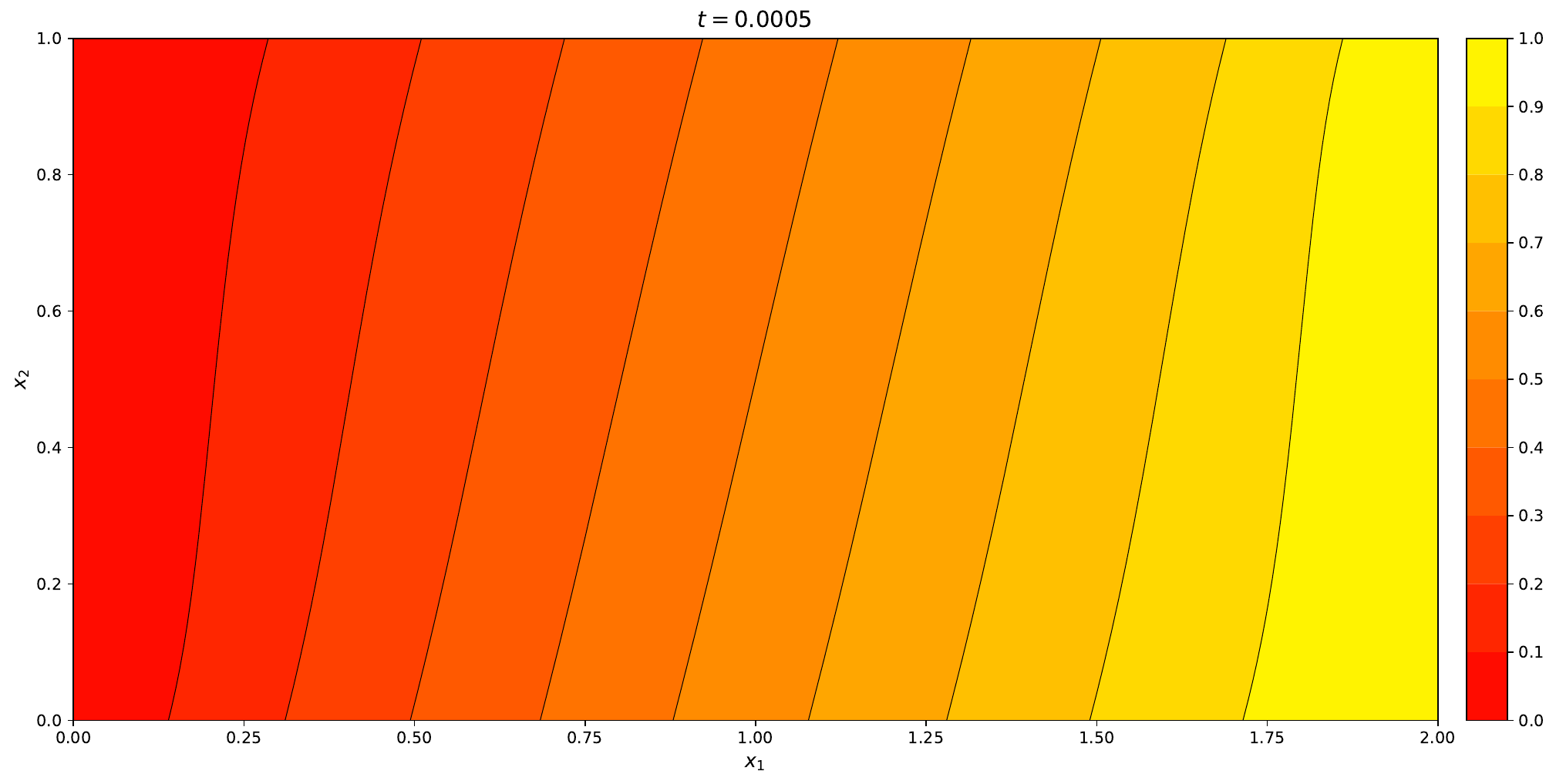}
	\includegraphics[width=0.49\linewidth]{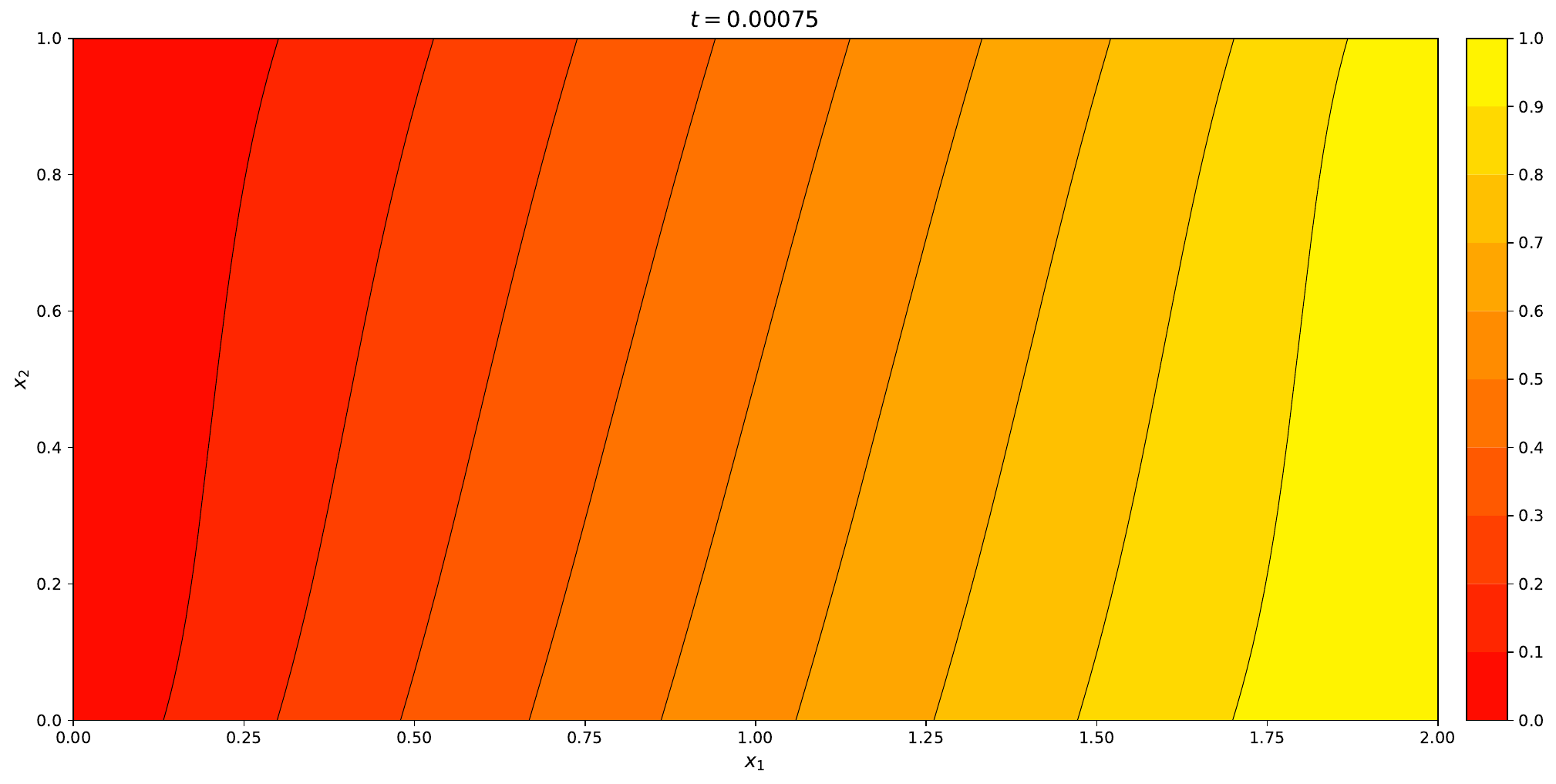}        
	\caption{Solution of the macroscale problem taking into account memory effects at different time instants.}
	\label{f-9}
\end{figure}

\clearpage

\section{Conclusions} \label{sec:6}

\begin{enumerate}[(1)]

\item An unsteady homogenization problem for flows in a porous medium is considered. At the microscale, the flow is described by the Stokes equations for an incompressible fluid, and at the macroscale by Darcy's law with memory. The components of the permeability tensor in the integro-differential equation are determined from the solution of an unsteady problem on the periodicity cell.

\item The proposed approach for accounting for memory effects is based on: (i) the standard determination of the effective permeability tensor from solutions of auxiliary steady-state boundary value problems on the periodicity cell at the microscale; (ii) approximation of the components of the tensor memory kernel as a sum of exponentials based on solving a partial Stokes spectral problem on the periodicity cell; and (iii) formulation of a local problem at the macroscale for an extended system of equations.

\item The computational implementation is based on the use of finite element spatial approximations and two-level difference schemes for time approximations. Unconditional stability estimates with respect to initial data are obtained for the solutions of the discrete problem, which are consistent with the corresponding estimates for the solution of the differential problem.

\item The performance of the developed computational algorithm and its efficiency are illustrated by calculations of a two-dimensional homogenization problem for unsteady flows in a porous medium.

\end{enumerate}

\end{document}